\documentclass[reqno,oneside,12pt]{amsart}

\usepackage[T1]{fontenc}
\usepackage{times,mathptm}
\usepackage{amssymb,epsfig,verbatim}

\theoremstyle{plain}
\newtheorem{thm}{Theorem}[section]
\newtheorem{theorem}[thm]{Theorem}

\newtheorem{lemma}[thm]{Lemma}

\newtheorem{proposition}[thm]{Proposition}
\theoremstyle{definition}
\newtheorem{remark}[thm]{Remark}

\newtheorem{example}[thm]{Example}

\newtheorem{conjecture}[thm]{Conjecture}

\newtheorem{question}[thm]{Question}

\numberwithin{equation}{section}
\newcommand{\sfD}{{\sf{D}}}
\newcommand{\sfT}{{\sf{T}}}
\newcommand{\UC}{{\rm{UC}}}
\newcommand{\Bir}{{\rm{Bir}}}
\newcommand{\Aut}{{\rm{Aut}}}
\newcommand{\bbPGL}{{\rm{PGL}}}
\newcommand{\GL}{{\rm{GL}}}

\newcommand{\R}{{\mathbf{R}}}
\newcommand{\C}{{\mathbf{C}}}
\newcommand{\Q}{{\mathbf{Q}}}
\newcommand{\Z}{{\mathbf{Z}}}
\newcommand{\NS}{{\rm{NS}}}

\newcommand{\bbP}{{\mathbb{P}}}
\newcommand{\Sph}{{\mathbb{S}}}

\newcommand{\BP}{{\mathbb P}}

\title [Calabi-Yau manifolds of Wehler type]{Birational automorphism groups and the movable cone theorem for Calabi-Yau manifolds of Wehler type via universal Coxeter groups}

\author{Serge Cantat and Keiji Oguiso}

\address{Serge Cantat, DMA, ENS-Ulm, UMR 8553 du CNRS \\
45 rue d'Ulm, F-75230 Paris cedex 05 
}
\email{serge.cantat@univ-rennes1.fr}
\address{
Keiji Oguiso, Department of Mathematics, Osaka University\\
Toyonaka 560-0043 Osaka, Japan and  Korea Institute for Advanced Study, Hoegiro 87, Seoul, 130-722, Korea} \email{oguiso@math.sci.osaka-u.ac.jp}

\dedicatory{Dedicated to Professor Yujiro Kawamata on the occasion of his
sixtieth birthday.}


\thanks{supported by JSPS Program 22340009 and by KIAS Scholar Program}

\begin{document}

\maketitle

\begin{abstract} Thanks to the theory of Coxeter groups, we produce the first family of Calabi-Yau manifolds $X$ of arbitrary dimension $n$, for which $\Bir(X)$ is infinite and the Kawamata-Morrison movable cone conjecture is satisfied. For this family, the movable cone is explicitly described; it's fractal nature is related to limit sets of Kleinian groups and to the Apollonian Gasket. Then, we produce explicit examples of (biregular) automorphisms with positive entropy  on some Calaby-Yau varieties.  \\

This version of the paper contains no picture ; see \cite{CaOg} for a file with pictures.  \end{abstract}

\section{Introduction}
\subsection{Coxeter groups}\label{intro:cox}
Coxeter groups (see eg. \cite{Hum}, \cite{Vi}) play a fundamental role in group theory.  Among all Coxeter groups generated by $N$ involutions, the {\bf{universal Coxeter group}} of rank $N$
\[
\UC(N) := \underbrace{       {\mathbf Z}/2{\mathbf Z} *  {\mathbf Z}/2{\mathbf Z} * \cdots * {\mathbf Z}/2{\mathbf Z}}_{N}\,\, ,
\]
where ${\mathbf Z}/2{\mathbf Z}$ is the cyclic group of order $2$, is the most basic one : there is no non-trivial relation between its $N$ natural generators, hence every Coxeter group is a quotient of some $\UC(N)$. 
The group $\UC(1)$ coincides with ${\mathbf Z}/2{\mathbf Z}$, and $\UC(2)$ is isomorphic to the infinite dihedral group; in particular, $\UC(2)$ is almost abelian in the sense that it contains a cyclic abelian subgroup $\mathbf Z$ as an index $2$ subgroup. For $N \ge 3$, $\UC(N)$ contains the free group ${\mathbf Z} * {\mathbf Z}$. 

Given a Coxeter diagram with $N$ vertices and its associated Coxeter group $W$, one can construct a real vector space $V$ of dimension $N$, together with a quadratic
form $b$, and a linear representation $W\to \GL(V)$ that preserves $b$ (see \S \ref{par:Part2}). We shall refer to this representation $W\to \GL(V)$ as the {\bf{geometric representation}} of $W$. The group 
$W$ preserves a convex cone $\sfT\subset V$, called the {\bf{Tits cone}}, that contains an explicit sub-cone $\sfD\subset \sfT$ which is a fundamental domain for the action of $W$ on $\sfT$. A precise description of 
the geometric representation and the invariant cone for $W=\UC(N)$ is given in Section \ref{par:UCN}. In particular, for $N\geq 3$, the quadratic form $b$ is non-degenerate and has signature $(1,N-1)$, and the fundamental domain $\sfD$ is a cone over a simplex.

\subsection{Wehler varieties}\label{par:introwehler}
Coxeter groups also appear in the theory of algebraic surfaces, mostly as hyperbolic reflection groups acting on N\'eron-Severi groups (see for instance \cite{Ni}, \cite{Bor}, \cite{CaDo}, \cite{Do}, \cite{Mc2}, \cite{To}). 

Let $n$ be a positive integer. Let $X$ be a smooth complex hypersurface of $(\bbP^1)^{n+1}$. The adjunction formula implies that the Kodaira
dimension of $X$ vanishes if and only if $X$ has multi-degree $(2,2,\ldots, 2)$, if and only if the canonical bundle of $X$ is trivial. 
These smooth hypersurfaces $X\subset (\bbP^1)^{n+1}$ of degree $(2,2,\ldots, 2)$ will be called {\textbf{Wehler varieties}} in what follows.

Let $X$ be such a Wehler variety. The $n+1$ projections $p_i\colon X\to ({\bbP^1})^{n}$ which are obtained
by forgetting one coordinate are ramified coverings of degree $2$; thus, for each index $i$, there is a birational transformation 
\[
\iota_i\colon X \dasharrow X
\]
that permutes the two points in the fibers of $p_i$. This provides a morphism  
\[
\Psi\colon\UC(n+1)\to \Bir(X)
\] 
into the group of birational transformations of $X$. 

\begin{remark}
Wehler surfaces and their variants  appear in the study of complex dynamics 
and arithmetic dynamics as handy, concrete examples (\cite{Ca}, \cite{Mc1}, \cite{Sil}). 
\end{remark}
\subsubsection{Curves and surfaces}\label{par:intro-curvesurfaces}
In dimension $n=1$, the hypersurface $X$ is an elliptic curve $\C/\Lambda$, and the image of the dihedral group $\UC(2)$ is 
contained in the semi-direct product of ${\Z}/2\Z$, acting by $z\mapsto -z$ on $X$, and $\C/\Lambda$, acting by translations.

For $n=2$, $X$ is a surface of type K3. Since the group of birational transformations of a K3 surface coincides with its
group of (biregular) automorphisms, the image of $\UC(3)$ is contained in $\Aut(X)$.  From a paper of Wehler, specifically from the last three lines in \cite{We}, one can deduce the following properties.
\begin{itemize}
\item The morphism $\Psi\colon \UC(3)\to \Aut(X)$ is injective;
\item if $X$ is generic, the dimension of its N\'eron-Severi group $\NS(X)$ is equal to $3$ and the linear representation of $\Aut(X)$ in $\GL(\NS(X)\otimes \R)$ is conjugate to the geometric representation of $\UC(3)$;
\item if $X$ is generic, the ample cone coincides with one connected
component of the set of vectors $u\in \NS(X)\otimes \R$ with self-intersection $u\cdot u > 0$.\end{itemize}

See Section~\ref{par:WSurf} for the proof of a slightly more general result. For simplicity, we refer to these statements as Wehler's theorem, even if they are
not contained in Wehler's paper.  

\subsubsection{Higher dimensional Wehler varieties} 
The aim of this  article is to extend the results of the previous section, in the birational and biregular ways, 
for Wehler varieties $X$ of higher dimension. For $n\geq 3$, $X$ is a {\textbf{Calabi-Yau manifold}} of dimension $n$. By definition, this means that $X$
is simply connected, there is no holomorphic $k$-form on $X$ for $0 < k < n$, and there is a nowhere vanishing holomorphic $n$-form $\omega_X$; hence, 
$$
H^0(X, \Omega_X^k) = 0, \; \forall k\in\{1, \ldots, n-1\}, \quad {\text{and}} \quad  H^0(X, \Omega_X^d) = \C\omega_X.
$$
Projective K3 surfaces are Calabi-Yau manifolds of dimension $2$. 

Before stating our main results, we need to recall the definition of 
various cones in $\NS\, (X)\otimes \R$ and to describe the movable cone conjecture of Morrison and Kawamata.  

\subsection{The movable cone conjecture} 
The precise formulation of the Kawamata-Morrison movable cone conjecture may be found in 
\cite{Ka2} (see Conjecture (1.12)). One can also consult \cite{Mo} for the biregular version of the conjecture and relation with mirror symmetry,
as well as \cite{To} for a more general version of these conjectures. In this short section, we introduce the main players and describe the conjectures.

\subsubsection{Cones (see \cite{Ka2})}\label{par:cones}
Let $M$ be a complex projective manifold. An integral divisor $D$ on $M$ is  {\textbf{effective}}
if the complete linear system $\vert D \vert$ is not empty; it is ${\mathbf Q}$-{\textbf{effective}} if there exists a positive integer 
$m$ such that $mD$ is effective; and it is {\textbf{movable}} if  $\vert D \vert$ has no fixed components, so that the base locus of
$\vert D\vert$ has codimension at least $2$. 

The N\'eron-Severi group $\NS(X)$ is the group of classes of divisors modulo numerical equivalence; given a ring ${\mathbf{A}}$, 
for example ${\mathbf{A}}\in \{\Z,\Q,\R,\C\}$, we denote by 
$\NS(X; {\mathbf{A}})$ the group $\NS(X)\otimes_{\Z} {\mathbf{A}}$. 

The {\textbf{effective cone}} ${\mathcal B}^e\, (M)\subset \NS(X;\R)$ 
is the convex cone generated by the classes of effective divisors.
The {\textbf{big cone}} ${\mathcal B}\, (M)$ (resp. the {\bf{ample cone}} ${\rm Amp}(X)$) is the convex cone generated by 
the classes of big (resp. ample) divisors. The {\bf{nef cone}} 
$\overline{{\rm Amp}}\, (M)$ is the closure of the ample cone, and ${\rm Amp}\, (M)$ 
is the  interior of the nef cone. 
Note that ${\mathcal B}\, (M) 
\subset {\mathcal B}^e\, (M)$ and the {\textbf{pseudo-effective cone}} $\overline{{\mathcal B}}\, (M)$ is the closure of both 
${\mathcal B}\, (M)$ and ${\mathcal B}^e\, (M)$.

The {\textbf{movable cone}}
$\overline{{\mathcal M}}\, (M)$ is the closure of the convex cone generated by 
the classes of movable divisors.  
The {\textbf{movable effective cone}} ${\mathcal M}^e\, (M)$ is defined to be 
$\overline{{\mathcal M}}\, (M) \cap {\mathcal B}^e\, (M)$. 

 Both the nef cone and the movable cone 
involve taking closures. 
So, {\it a priori}, $\overline{{\mathcal M}}\, (M)$ is not necessarily a subset of ${\mathcal B}^e\, (M)$; similarly, $\overline{{\rm Amp}}\, (M)$ is not necessarily a subset of ${\mathcal B}^e\, (M)$ nor of ${\mathcal M}^e\, (M)$, while it is always true that 
$\overline{{\rm Amp}}\, (M) \subset \overline{{\mathcal M}}\, (M)$. 

\subsubsection{The cone conjecture}
Assume that $M$ is a Calabi-Yau manifold. All birational transformations of $M$ are isomorphisms in codimension
one (see \S \ref{par:wv} below); in other words, the group of birational transformations coincides with the group of {\textbf{pseudo-automorphisms}}.
Thus, if $g$ is an element of $\Bir(M)$ and $D$ is movable (resp. effective, $\mathbf Q$-{effective}), then 
$g^*(D)$ is again movable (resp. effective, $\mathbf Q$-{effective}). As a consequence, $\Bir\, (M)$ naturally acts on the
three cones $\overline{{\mathcal M}}\, (M)$, ${\mathcal B}^e\, (M)$ and ${\mathcal M}^e\, (M)$.

The abstract version of the Morrison-Kawamata movable cone conjecture is the following:

\begin{conjecture}[Morrison and Kawamata]\label{morrison-kawamata}
The action of $\Bir\,(M)$ 
on the movable effective cone ${\mathcal M}^{e}(M)$ has a finite rational polyhedral cone 
$\Delta$ as a fundamental domain. Here $\Delta$ is called 
a fundamental domain if
\[
g^*(\Delta^{\circ}) \cap \Delta^{\circ} = \emptyset\, ,\, 
\forall g\in \Bir(X) \, \,  {\text{with}} \, \, g^* \not= {\text{Id}}\,\, ,
\]
where $\Delta^{\circ}$ is the interior of $\Delta$, and  
\[
\Bir\, (M) \cdot \Delta = {\mathcal M}^e\, (M)\, .
\]
\end{conjecture}

This conjecture holds for log K3 surfaces (\cite{To}) and abelian varieties (\cite{Ka2}, \cite{PS}), and its relative version has been verified for fibered Calabi-Yau threefolds (\cite{Ka2}). The conjecture is also satisfied for several interesting examples of Calabi-Yau threefolds: See for instance  \cite{GM} for the biregular situation and \cite{Fr} for the birational version. A version of the conjecture is proved in \cite{Mar} for compact hyperk\"ahler manifolds.

\subsection{The movable cone conjecture for Wehler varieties}
Our first main result is summarized in the following statement.

\begin{theorem}\label{main1}
Let $n \ge 3$ be an integer. Let $X$ be a generic hypersurface of multi-degree $(2, \ldots, 2)$ in $({\bbP^1})^{n+1}$. Then,
\begin{enumerate}
\item the   automorphism group  $\Aut(X)$ is trivial, i.e. $\Aut\, (X) = \{{\rm{Id}}_X \}$; 

\item The morphism $\Psi$ that maps each generator $t_j$ of $\UC(n+1)$ to the involution $\iota_j$ of $X$ is an
isomorphism $\Psi\colon  \UC(n+1)\to \Bir\,(X)$;

\item  $X$ satisfies Conjecture \ref{morrison-kawamata}: The cone $\overline{{\rm Amp}}\, (X)$ is a fundamental domain for the action of $\Bir\,(X)$ on the movable effective cone ${\mathcal M}^e(X)$; this fundamental domain is a cone over a simplex. 
\end{enumerate}
More precisely, there is a linear conjugacy between the (dual of the) geometric representation of $\UC(n+1)$ and the representation of $\Bir(X)$ on $\NS(X)$ that maps equivariantly the Tits cone $\sfT$ to the movable cone and the
fundamental cone $\sfD$ to the nef cone $\overline{{\rm Amp}}\, (M)$.
\end{theorem}

The conjugacy described in the second and third assertions enables us to describe more precisely the geometry of the movable cone. 
First, one can list explicitly all rational points on the boundary of the movable cone and prove that {\sl{every rational point on the boundary of the (closure of the) movable cone is movable}}, a result that is not predicted by Kawamata's conjecture. Second, one can {\sl{draw pictures of this cone}}, 
and show how it is related to Kleinian groups and the Appollonian Gasket. In particular, we shall see that
{\sl{the boundary of the movable cone has a fractal nature when $n\geq 3$}}; it is not smooth, nor polyhedral. 

 To our best knowledge, this theorem is the first non-trivial result in which the movable cone conjecture is checked for non-trivial examples of Calabi-Yau manifolds in dimension $\ge 4$. 
Our proof is a combination of recent important progress in the minimal model theory in higher dimension due to 
Birkar, Cascini, Hacon, and McKernan (\cite{BCHM}) and to Kawamata (\cite{Ka3}) and of classical results concerning Coxeter groups and Kleinian groups
(Theorems \ref{coxeter} and \ref{universalcoxeter} below).

\subsection{Automorphisms with positive entropy}
As already indicated by Theorem \ref{main1}, in higher dimensional algebraic geometry, birational transformations are more natural, and in general easier to find, than regular automorphisms. Nevertheless, it is also of fundamental interest to find non-trivial regular automorphisms of higher dimensional algebraic varieties. The following result describes families of Calabi-Yau manifolds in all even dimensions (\cite{OS}, Theorem 3.1): {\sl{
Let $Y$ be an Enriques surface and ${\rm Hilb}^n (Y)$ be the Hilbert scheme of $n$ points on $Y$, where $n \ge 2$. Let 
\[
\pi : \widetilde{{\rm Hilb}^n(Y)} \to 
{\rm Hilb}^n (Y)
\] 
be the universal cover of ${\rm Hilb}^n (Y)$. Then $\pi$ 
is of degree $2$ and $\widetilde{{\rm Hilb}^n(Y)}$ is a Calabi-Yau manifold of dimension $2n$. }}

In Section \ref{par:main2}, we prove that the automorphism group of $\widetilde{{\rm Hilb}^n(Y)}$ can be very large, and may contain element with positive topological entropy. This is, again, related to explicit Coxeter groups. For instance, we prove the following result. 

\begin{theorem}\label{main2}
Let $Y$ be a generic Enriques surface. Then, for each $n \ge 2$, the biregular automorphism group of $\widetilde{{\rm Hilb}^n(Y)}$ contains a subgroup isomorphic to the universal Coxeter group $\UC(3)$, and this copy of $\UC(3)$ contains an automorphism with positive entropy. 
\end{theorem}

The Calabi-Yau manifolds of Theorems \ref{main1} and \ref{main2} are fairly concrete. We hope that these examples will also provide non-trivial handy examples for complex dynamics (\cite{DS}, \cite{Zh}) and arithmetic dynamics (\cite{Sil}, \cite{Kg}) in higher dimension.


\section{Universal Coxeter groups and their geometric representations}\label{par:Part2}


In this section we collect a few preliminary facts concerning the universal Coxeter group on $N$ generators, 
and describe the geometry of its Tits cone. 

\subsection{Coxeter groups (see \cite{Hum})}
\subsubsection{Definitions}
Let $W$ be a group with a finite set of generators $S = \{s_j \}_{j=1}^{N}$. The pair $(W, S)$ is a {\bf{Coxeter system}} if there
are integers $m_{ij}\in {\mathbf Z}_+\cup \{\infty\}$ such that
\begin{itemize}
\item $W = \langle s_j \in S\, \vert\, (s_is_j)^{m_{ij}} = 1 \rangle$ is a presentation of $W$,
\item $m_{jj} =1$, i.e., $s_j^2 = 1$ for all $j$,
\item $2 \le m_{ij} = m_{ji} \le \infty$ when $i \not= j$ (here $m_{ij} = \infty$ means that $s_is_j$ is of infinite order).
\end{itemize} 
A group $W$ is called a {\bf Coxeter group} if $W$ contains a finite 
subset $S$ such that $(W, S)$ forms a Coxeter system. 

Let  $ \UC(N) $ be the free product of $N$ cyclic groups ${\Z}/2\Z$ of order $2$, as in Section \ref{intro:cox}. The group 
$\UC(N)$ is generated by $N$ involutions $t_j$, $1\leq j\leq N$, with no non-obvious relations between them. With this set of 
generators, $\UC(N)$ is a Coxeter group with $m_{ij}=\infty$ for all $i\neq j$.

If $(W, S)$ is a Coxeter system with $\vert S \vert = N$, there is a unique surjective homomorphism 
$\UC(N) \to W$ that maps $ t_j $ onto  $s_j$; its  kernel is the minimal normal subgroup containing $\{ (t_it_j)^{m_{ij}} \}$. In this 
sense, 
the group $\UC(N)$ is  universal among all Coxeter groups with $N$ generators; thus, we call $\UC(N)$ (resp. $(\UC(N), \{t_j\}_{j=1}^{N})$) the {\textbf{universal Coxeter group}}
(resp. the universal Coxeter system) {\textbf{of rank}} $N$. 

\subsubsection{Geometric representation}\label{par:geomrep}
Let $(W, \{s_j\}_{j=1}^{N})$ be a Coxeter system with 
$(s_is_j)^{m_{ij}} = 1$, as in the previous paragraph.  An $N$-dimensional real vector space $V = \oplus_{j=1}^{N} \R\alpha_j$ and a bilinear form $b(*,**)$ on $V$ are associated to these data;  the quadratic form $b$ is defined by  its values on the basis $(\alpha_j)_{j=1}^{N}$:
\[
b(\alpha_i, \alpha_j) = -\cos \frac{\pi}{m_{ij}}\,\, .
\]
Then, as explained in \cite{Hum} (see the Proposition page 110), there is a well-defined linear representation $\rho : W \to \GL(V)$, which maps each generator $s_j$ to the symmetry
\[
\rho(s_j)\colon \lambda\in V \mapsto \lambda - 2b(\alpha_j, \lambda)\alpha_j\,\, .
\]
The representation $\rho$ is the {\bf{geometric representation}} of the Coxeter system $(W,S)$. The following theorem (see \cite{Hum}, Page 113, Corollary) is one of the fundamental results on Coxeter groups:

\begin{theorem}\label{coxeter} The geometric representation $\rho$ of a Coxeter system $(W, S)$ is faithful. In particular, all Coxeter groups are linear.  
\end{theorem}

\subsubsection{Tits cone} \label{par:titsconegene}
The dual space $V^*$ contains two natural convex cones, which we now describe. Denote the dual representation of
$W$ on $V^*$ by $\rho^*$.
The first cone $\sfD\subset V^*$ is the intersection of the half-spaces 
\[
\sfD^+_i=\left\{ f\in V^*\vert \, \, f(\alpha_i)\geq 0 \right\}.
\]
It is a closed convex cone over a simplex of dimension $N-1$; the facets of $\sfD$ are the intersections $\sfD\cap \sfD^+_i$. 

\begin{remark}
Since $\rho(s_i)$ maps
$\alpha_i$ to its opposite, the action of $\rho^*(s_i)$ on $V^*$ exchanges $\sfD^+_i$ and $-\sfD^+_i$. 
This is similar to the behavior of the relatively ample classes in the flopping diagram.
\end{remark}

The second cone, called the {\bf {Tits cone}}, is the union $\sfT$ of all images $\rho^*(w)(\sfD)$, where $w$ describes $W$. 

\begin{theorem}[see \cite{Hum}, \S 5.13, Theorem on page 126]\label{Tits-Cone}
The Tits cone $\sfT\subset V$ of a Coxeter group $W$ is a convex cone. It is invariant under the action of $W$ on $V$ and $\sfD$ is a fundamental domain for this action. 
 \end{theorem}

Given $J\subset S$, consider the subgroup $W_J$ of $W$ generated by the elements of $J$. Define $\sfD(J)$ by 
\[
\sfD(J) = \left( \cap_{j\in J} \{f\in V^* \vert \, f(\alpha_j)=0 \}  \right)\cap \left( \cap_{i\notin J} \{f\in V^* \vert \, f(\alpha_i) >0\}  \right);
\] 
hence, each $\sfD(J)$ is the interior of a face of dimension $N-\vert J\vert$. Then $W_J$ is a Coxeter group; it coincides with
the stabilizer of every point of $\sfD(J)$ in $W$. 

\subsection{The universal Coxeter group}\label{par:UCN}

We now study the  Coxeter group $\UC(N)$. The vector space of its geometric representation, the Tits cone and its
fundamental domain are denoted with index $N$: $V_N$, $\sfT_N$, $ \sfD_N$, etc.

\subsubsection{The linear representation}\label{matrix}
With the basis $(\alpha_j)_{j=1}^{N}$, identify $V_N$ to $\R^N$, and $\GL(V)$ to $\GL_N(\R)$.
Let $M_{N,j}$ ($1 \le j \le N$) be the $N \times N$ matrices with integer coefficients, defined by: 
\begin{equation}\label{eq:matrix}
M_{N,j} = 
\left(\begin{array}{rrrrrrrrr}
1 & 0 & \ldots & 0 & 2  & 0 & \ldots & 0\\
0 & 1 & \ldots & 0 & 2  & 0 & \ldots & 0\\
\vdots & \vdots & \ddots & \vdots & \vdots & \vdots & \ldots & \vdots\\ 
0 & 0 & \ldots & 1 & 2 & 0 & \ldots & 0\\
0 & 0 & \ldots & 0 & -1 & 0 & \ldots & 0\\
0 & 0 & \ldots & 0 & 2 &  1 & \ldots & 0\\
\vdots & \vdots & \ddots & \vdots & \vdots & \vdots & \ddots & \vdots\\
0 & 0 & \ldots  & 0 &2 & 0 & \ldots & 1 
\end{array} \right)\,\, ,
\end{equation}
where $-1$ is the $(j,j)$-entry. For instance, 
\[
M_{3,1} = \left(\begin{array}{rrr}
-1 & 0 & 0\\
2 & 1 & 0\\
2 & 0 & 1
\end{array} \right)\,\, ,\,\, M_{3,2} = \left(\begin{array}{rrr}
1 & 2 & 0\\
0 & -1 & 0\\
0 & 2 & 1
\end{array} \right)\,\, ,\,\, M_{3,3} = \left(\begin{array}{rrr}
1 & 0 & 2\\
0 & 1 & 2\\
0 & 0 & -1
\end{array} \right)\,\, .
\] 

\begin{theorem}\label{universalcoxeter} The geometric representation $\rho$ of the universal Coxeter system $(\UC(N), \{t_j\}_{j=1}^{N})$ 
of rank $N$ is given by $\rho(t_j) = M_{N, j}^{t}$, where $M_{N, j}^{t}$ is the transpose of $M_{N, j}$. In particular, in $\GL_N(\R)$, we obtain
\begin{eqnarray*}
\langle M_{N, j}\, , \,  1 \le j \le N \rangle &  = &\langle M_{N, 1} \rangle * \langle M_{N, 2} \rangle * \cdots * \langle M_{N, N} \rangle \\
& = & \rho(\UC(N))\\
&\simeq & \UC(N).
\end{eqnarray*}
\end{theorem}
\begin{proof} 
By definition, $m_{jj} = 1$ and $m_{ij} = \infty$ ($i\not= j$) for the universal Coxeter system. Hence 
$$\rho(t_j)(\alpha_i) = \alpha_i + 2\alpha_j\,\, (i\not= j)\,\, ,\,\, 
\rho(t_j)(\alpha_j) = -\alpha_j\,\, ,$$
i.e., the matrix representation of $\rho(t_j)$ in  the basis $(\alpha_j)_{j=1}^{N}$ is $M_{N, j}^{t}$. 
\end{proof}

\subsubsection{The quadratic form}\label{par:quadform}
Let $b_N$ denote the {\sl{opposite}} of the quadratic form defined in Section \ref{par:geomrep}. 
Its matrix $B_N$, in the basis $(\alpha_i)_{i=1}^{i=N}$ is the integer matrix with coefficients $-1$ on the diagonal and $+1$ for all 
remaining entries.

When $N=1$, $b_N$ is negative definite. When $N=2$, $b_N$ is degenerate: $b(u,u)=(x-y)^2$ for all $u=x\alpha_1+y\alpha_2$ in $V$. 
For $N\geq 3$, the following properties are easily verified (with $\parallel v \parallel_{euc}$ the usual euclidean norm): 
\begin{enumerate}
\item The vector $u_N=\sum_{i=1}^{N} \alpha_i$ is in the positive cone; more precisely 
\[
b_N(u_N,u_N)=N(N-2).
\]
\item $b_N(v,v)= -2\parallel v\parallel_{euc}^2=-2\sum x_i^2$ for all vectors $v=\sum_{i=1}^Nx_i \alpha_i$ of the orthogonal
complement $u_N^\perp=\{v=\sum_i x_i \alpha_i\vert \, \, \sum_i x_i=0\}$.
\item The signature of $b_N$ is $(1,N-1)$.
\end{enumerate}
In what follows, $N\geq 3$ and $u_N$ denotes the vector $\sum_i\alpha_i$. A vector $w=au_N+v$, with $v$ in $u_N^\perp$, 
is isotropic if and only if 
\[
N(N-2) a^2= 2 \parallel v\parallel_{euc}^2.
\]
Thus, if $(\beta_1, \ldots, \beta_{N-1})$ is an orthonormal basis of $u_N^\perp$ and $\beta_N=u_N$, then the isotropic cone
is the cone over a round sphere; its equation is
\[
(N(N - 2)/2) y_n^2= \sum_{i=1}^{N-1} y_i^2
\] 
for $v=\sum_j y_i \beta_j$.
The vectors $\alpha_i+\alpha_j$, with $i\neq j$, are isotropic vectors. 

\begin{example}\label{eg:N=3}
For $N=3$, consider the basis $((0,1,1),(1,0,1),(1,1,0))$. The matrix of $b_3$ in this basis has coefficients $0$ along the diagonal, 
and coefficients $1$ on the six remaining entries. 
\end{example}

\subsubsection{The Tits cone $\sfT_N$} 
To understand the Tits cone of $\UC(N)$, one can identify $V_N$ to its dual $V_N^*$ by the duality
given by the non-degenerate quadratic form $b_N$; with such an identification, $\sfD_N$ is the set of vectors $w$ such that 
$b_N(w,\alpha_i)\geq 0$ for all $i$, and  $\sfT_N$ becomes a convex cone in $V_N$. The half-space $\sfD^+_{N,i}$ is the set of vectors $w$
such that $b_N(w, \alpha_i)\geq 0$ and its boundary is the hyperplane $\alpha_i^\perp$. The convex cone $\sfD_N$ is the convex hull
of its $N$ extremal rays ${\mathbf{R}}_+ c_j$, $1\leq j\leq N$, where 
\begin{eqnarray*}
c_N & = & -(N-2)\alpha_N + \sum_{i=1}^N \alpha_i \\
& = & (1,1, \ldots, 1, -(N-3))
\end{eqnarray*}
and the $N-1$ remaining $c_j$ are obtained from $c_N$ by permutation of the coordinates. For all $N\geq 3$, and all indices $j$, one obtains
\[
b_N(c_j,c_j)=-2 (N-2)(N-3).
\]

\begin{example}
For $N=3$, $(c_1,c_2,c_3)$ is the isotropic basis already obtained in Example \ref{eg:N=3}. For $N=4$, one gets 
the basis $(-1,1,1,1)$, $(1,-1,1,1)$, $(1, 1, -1, 1)$, $(1,1,1,-1)$, with $b_4(c_j, c_j)=-4$.\end{example}




\subsubsection{A projective view of $\, \sfD_N$}
To get some insight in the geometry of the Tits cone, one can draw its projection in the real projective space ${\BP}(V_N)$, 
at least for small values of $N$. We denote the projection of a non-zero vector $v$ in $\BP(V_N)$ by $[v]$.

To describe ${\BP}(\sfT_N)$ and $\BP(\sfD_N)$, denote by $\Sph_N$ the projection of the isotropic cone (it is a round sphere, as described in 
\S~\ref{par:quadform}). Consider the projective line $L_{ij}$ through the two points $[c_i]$ and $[c_j]$ in $\BP(V_N)$; this line is the projection of
the plane ${\text{Vect}}(c_i,c_j)$. For example, with $(i,j)=(N-1,N)$, this plane is parametrized by $sc_{N-1}+tc_N$ with $s$ and $ t$ in ${\mathbf{ R}}$, and its intersection with the isotropic cone corresponds to parameters $(s,t)$ such that 
\[
(N-3) (s^2+t^2) =2 st.  
\]
Hence, we get the following behaviour: 
\begin{itemize}
\item If $N=3$, $L_{ij}$ intersects the sphere $\Sph_N$ transversally at $[c_i]$ and $[c_j]$.
\item If $N=4$, $L_{ij}$ is tangent to $\Sph_N$ at $[c_i+c_j]$.
\item If $N\geq 5$, $L_{ij}$ does not intersect $\Sph_N$.
\end{itemize}
Similarly, one shows that the faces $\alpha_i^\perp\cap \alpha_j^\perp$ of $\sfD_N$ of codimension $2$ intersect 
the isotropic cone on the line ${\mathbf{R}}(\alpha_i+\alpha_j)$: Projectively, they correspond to faces of $\BP({\sfD_N})$ that
intersect the sphere $\Sph_N$ on a unique point. The faces $\alpha_i^\perp\cap \alpha_j^\perp\cap\alpha_k^\perp$ of codimension $\geq 3$
do not intersect the isotropic cone (i.e. do not intersect $\Sph_N$ in $\BP(V_N)$).

In dimension $N=3$,  $\sfD$ is a triangular cone with isotropic extremal rays and $\BP(\sfD_3)$ is a triangle. 



\subsubsection{A projective view of the Tits cone $\sfT_3$ (see \cite{Mag})}\label{par:T3}
The cone $\sfT_N$ is the union of all images $\rho(w)(\sfD_N)$, for $w$ in $\UC(N)$. 
When $N=3$, the sphere $\Sph_3$ is a circle, that bounds a disk. This disk can be identified to 
the unit disk in $\C$ with its hyperbolic metric, and the group $\UC(3)$ to a discrete subgroup 
of the isometry group $\bbPGL_2(\R)$. Up to conjugacy, $\UC(3)$ is the 
congruence subgroup 
\[
\{ M\in \bbPGL_2(\Z)\, \vert \; M={\text{Id}}_2 \; {\text{mod}}(2)\}.
\]
In particular, $\UC(3)$ is a non-uniform lattice in the Lie group $\bbPGL_2(\R)$ (i.e. in the orthogonal 
group of $b_3$).
The fundamental domain $\BP(\sfD_3)$ is a triangle with vertices on the circle $\Sph_3$. 
Its orbit $\BP(\sfT_3)$ under $\UC(3)$ is the union of
\begin{itemize}
\item the projection of the positive cone $\{ w \in V\vert \, \, b(w,u)>0 \, \, {\text{ and }} \, \, b(w,w)>0\}$ in $\BP(V)$,
\item  the set of rational points $[w]\in \Sph_3$ where
$w$ describes the set of isotropic vectors with integer coordinates. 
\end{itemize}
All rational points of $\Sph_3$ can be mapped to one of the vertices of $\BP(\sfD_3)$ by the action of $\UC(3)$.

\subsubsection{The limit set}\label{par:limitset}

The group $\UC(N)$ preserves the quadratic form $b_N$ and this form is non-degenerate, of signature 
$(1,N-1)$. Thus, after conjugacy by an element of $\GL_N(\R)$, $\UC(N)$ becomes a discrete
subgroup of ${\rm O}_{1,N-1}(\R)$. The {\bf{limit set}} of such a group is the minimal compact 
subset of the sphere $\Sph_N$ that is invariant under the action of $\UC(N)$; we shall denote the limit set 
of $\UC(N)$ by  $\Lambda_N$.
This set coincides with (see \cite{Rat}, chap. 12)
\begin{itemize}
\item the closure of the points $[v]$ for all vectors $v\in V_N$ which are eigenvectors of at least one element $f$ 
in $\UC(N)$ corresponding to an eigenvalue $>1$;
\item the intersection of $\Sph_N$ with the closure of the orbit $\UC(N)[w]$, for any given $[w]$ such that $b_N(w,w)\neq 0$.
\end{itemize} 
The convex hull ${\sf{Conv}}(\Lambda_N)$ of the limit set is invariant under $\UC(N)$, and is contained in the closed ball enclosed in $\Sph_N$. 
The dual of ${\sf{Conv}}(\Lambda_N)$ with respect to the quadratic form $b_N$ is also a closed invariant convex set. One can show that this 
convex set coincides with the closure of $\BP(\sfT_N)$, and corresponds to the maximal invariant and strict cone in $V_N$ (see \cite{B}, \S 3.1); we shall not use this 
fact (see Section \ref{par:proofmovable} for a comment).

\begin{remark}
For $N=4$, ${\rm O}(b_4)$ is isogeneous to $\bbPGL_2(\C)$ and $\UC(4)$ is conjugate to a Kleinian group (see \cite{Bea}). 
\end{remark}

\begin{remark}
Note that $\UC(3)$ is a lattice in the Lie group ${\rm O}(b_3)$ while, for $N\geq 4$, the fundamental domain  $\BP(\sfD_N)$ contains points of the sphere $\Sph_N$ in its
interior and the Haar measure of ${\rm O}(b_N)/\UC(N)$ is infinite.
\end{remark}

\subsubsection{A projective view of the Tits cone $\sfT_4$}
Let us fix a vertex, say $c_N$, of $\sfD_N$. Its stabilizer is  the subgroup $\UC(N)_{\{N\}}$ of $\UC(N)$ (cf. \S~\ref{par:titsconegene}). It acts on $c_N^\perp$ and the orbit of
$\sfD_N\cap c_N^\perp$ under the action of $\UC(N)_{\{N\}}$ tesselates the intersection $\sfT_N\cap c_N^\perp$; moreover, up to conjugacy 
by a linear map, $\sfT_N\cap c_N^\perp$,
together with its action of $\UC(N)_{\{ N \}}$, is equivalent to the Tits cone in dimension $N-1$, together with the action of $\UC(N-1)$. 
Hence, {\sl{if one looks at $\BP(\sfT_N)$ from a vertex of $\BP(\sfD_N)$, one obtains a cone over $\BP(\sfT_{N-1})$}}. 

Let us apply this fact to the case $N=4$. Then, $V_4$ has dimension $4$, $\BP(V_4)$ has dimension $3$, and the sphere $\Sph_4$ is 
a round sphere in ${\mathbf{R}}^3$ (once the hyperplane $[u_4^\perp]$ is at infinity). In particular, $\UC(4)$ acts by conformal 
transformations on this sphere; thus, $\UC(4)$ is an example of a Kleinian group, i.e. a discrete subgroup
of ${\rm O}_{1,3}(\R)$ (see \S \ref{par:limitset}).



Viewed from the vertex $[c_4]\in \BP(V_4)$, the convex set $\BP(\sfT_4)$ looks like 
a cone over $\BP(\sfT_3)$. More precisely, the intersection of $[c_4^\perp]$ with the sphere $\Sph_4$ is a circle, and the orbit of $\BP(\sfD_4)\cap [c_4^\perp]$ 
under the stabilizer $\UC(4)_{\{4\}}$ tesselates the interior of this circle (see \S \ref{par:titsconegene}).  All segments that connect
$[c_4]$ to a rational point of this circle are contained in $\BP(\sfT_4)$. Thus, $\BP(\sfT_4)$ contains a ``chinese hat shell'' (Calyptraea chinensis),
tangentially glued to the sphere $\Sph_4$, with vertex $[c_4]$. The  picture is similar in a neighborhood of $[c_1]$, $[c_2]$, and $[c_3]$. 

Since $\UC(4)$
acts by projective linear transformations on $\BP(V_4)$, an infinite number of smaller and smaller chinese hat shells are glued to $\Sph_4$, 
with vertices on the orbits of the $[c_i]$: the accumulation points of these orbits, and of the "chinese hats" attached to them, converge towards
the limit set of the Kleinian group $\UC(4)$.

 To get an idea of the convex set $\BP(\sfT_4)$, one needs to describe the sequence of circles
along which the shells are glued: This sequence is made of circles on the sphere $\Sph_4$, known as the Apollonian Gasket. 
The closure of the union of all these circles is a compact subset of $\Sph_4$ that is invariant under the action of $\UC(4)$; as such, it coincides
with the limit set of the Kleinian group $\UC(4)$, {\sl{i.e.}} with the unique minimal $\UC(4)$-invariant compact subset of $\Sph_4$. Its Hausdorff dimension has been computed by McMullen (see \cite{Mc3}):
\[
{\text{H.-dim(Apollonian Gasket)}}=1.305688...
\]

\begin{remark}
Up to conjugacy in the group of conformal transformations of the Riemann sphere, there is a unique configuration of tangent circles with the 
combinatorics of the Apollonian Gasket. Thus this circle packing is ``unique''.
\end{remark}

\begin{proposition}
The projective image of the Tits cone $\sfT_4$ is a convex set $\BP(\sfT_4)\subset \BP(V_4)$. Its closure
${\overline{\BP(\sfT_4)}}$ is a compact convex set. 
Let ${\sf{Ex}}({\overline{\BP(\sfT_4)}})$ be the set of extremal points of ${\overline{\BP(\sfT_4)}}$. This set
is $\UC(4)$-invariant, and the set of all its accumulation points coincides with the limit set $\Lambda(4)$ of $\UC(4)$ in $\Sph_4$. The
Hausdorff dimension of $\Lambda(4)$ is approximately equal to $1.305688$.
\end{proposition}

In higher dimension the set $\BP(\sfT_N)$ contains a solid cone with vertex 
$[c_N]$ and with basis $\BP(\sfT_N)\cap [c_N^\perp]$ equivalent to $\BP(\sfT_{N-1})$; the picture is similar
around each vertex $[c_i]$, and the convex set $\BP(\sfT_N)$ is the union of the interior of the sphere
and the orbits of these chinese hat shells (see below for a precise definition). 
Thus, the complexity of $\sfT_N$ increases with $N$. 

\subsubsection{Rational points}\label{par:Cox-ratio-points}
This section may be skipped on a first reading; it is not needed to prove the Kawamata-Morrison conjecture in the Wehler examples, but is useful in order to provides stronger results concerning the rational points on the boundary of the movable cone. 
The following two propositions describe the rational points of the boundary of $\BP(\sfT_N)$; it shows that these rational points are the obvious ones.

\begin{proposition}\label{pro:rat-pts1}
Let $N\geq 3$ be an integer.
Let $[v]$ be a rational point of the boundary of the convex set $\BP(\sfT_N)$.
Then $[v]$ is of the form $[\rho(w)(v')]$ where 
\begin{itemize}
\item $[v']$ is a rational point of the boundary of $\BP(\sfT_N)$;
\item $[v']$ is a point of $\BP(\sfD_{N})$.
\end{itemize}
\end{proposition}

Before starting the proof, define the {\bf{chinese hat shell}} of $\BP(\sfT_N)$ with vertex $[c_N]$ as the set of points
of $\BP(\sfT_N)$ contained in the convex set generated by $[c_N]$ and $\BP(\sfT_{N})\cap [c_N^\perp]$; 
this subset of $\BP(\sfT_N)$ is a solid cone with vertex $[c_N]$ and basis equivalent to $\BP(\sfT_{N-1})$; 
it coincides with the projection of the subset 
\[
\{v\in V\; \vert \; v\in \sfT_N \; {\text{and}} \; b_N(c_N,v)>0\}= \{v=\sum_i a_i \alpha_i \; \vert \; v\in \sfT_N \; {\text{and}} \; a_N\leq 0\}.
\]
Each vertex $[c_i]$ of $\BP(\sfT_N)$ generates a similar shell; in this way, we get $N$ {\bf{elementary shells}}, one per vertex $[c_i]$. The image of such a shell by an element $\gamma$ of $\UC(N)$ is, by definition, the chinese hat shell with vertex $[\gamma(c_i)]$. 

\begin{remark}[see \cite{Rat}] Let $H$ be one of these shells; denote by $s$ its vertex. If $(\gamma_n)$ is a sequence of elements of $\UC(N)$ going to infinity all accumulation points of $(\gamma_n(s))$ are contained in the limit set $\Lambda_N\subset \Sph_N$ of $\UC(N)$. Thus,  the sequence $(\gamma_n(H))$ is made of smaller and smaller shells and its accumulation points are also contained in the limit set $\Lambda_N$. Conversely, every point of $\Lambda_N$ is the limit of such a sequence $(\gamma_n(H))$.\end{remark}

\begin{remark}\label{rem:Tits-shells}
The Tits cone $\sfT_N$ is the orbit of $\sfD_N$. Thus, an easy induction on $N$ based on the previous remark, shows that {\sl{the boundary points of $\bbP(\sfT_N)$ are contained in the union of the shells and of the limit set $\Lambda_N$}}. \end{remark}

\begin{proof} 
We prove the statement by induction on $N\geq 3$. The case $N=3$ has already been described previously. 
Let $[v]$ be a rational point. One can assume that $[v]$ is the projection of a vector $v\in V$ with integer, relatively prime, coordinates.

If $[v]$ is in a chinese hat shell, the conclusion follows from the induction hypothesis. If not, $[v]$ is in the sphere $\Sph_N$ (cf. Remark~\ref{rem:Tits-shells}). Thus, $v=\sum a_i \alpha_i$ satisfies the  equation of $\Sph_N$; this can be written
\begin{equation}  \label{eq:sphere2}
\sum_{i=1}^N a_i \left(\sigma(v) -2 a_i\right)=0
\end{equation}
where $\sigma(v)=\sum_{j=1}^N a_j$. The condition that assures that $[v]$ is in the chinese hat shell with vertex $[c_i]$ reads $a_i\leq 0$. 
Thus, one can assume $a_i\geq 1$ for all indices $i$, because the $a_i$ are integers. From Equation \eqref{eq:sphere2}, one deduces
that $\sigma(v)-2a_i\leq -1$ for at least one index $i$, say for $i=1$. Apply the involution $\rho(t_1)$. Then, the first coordinate
$a'_1$ of $\rho(t_1)(v)$ is equal to $-3a_1 + 2\sigma(v)$, while the other coordinates remain unchanged: $a'_j=a_j$ for $j\geq 2$. 
Hence, 
\[
a'_1-a_1= -4 a_1 + 2\sigma(v)= 2 (-2a_1 + \sigma(v))\leq -2,
\]
so that $a'_1-a_1$ is strictly negative. 

Iterating this process a finite number of times, the sum $\sigma(\cdot)$ determines a decreasing sequence of positive integers. 
Consequently, in a finite number of steps, one reaches the situation where a coefficient $a_i$ is negative, which means that the orbit of $[v]$ under
the action of $\UC(N)$ falls into one of the $N$ elementary chinese hat shells, as required. 
\end{proof}

Assume, now that $[v_1], \ldots [v_{l+1}]$ are rational boundary points of $\BP(\sfT_N)$. If the convex set 
\[
C={\sf{Conv}}([v_1], \ldots, [v_{l+1}])
\]
has dimension $l$ and is contained in the boundary of $\BP(\sfT_N)$, we say that $C$ is 
a {\bf{rational boundary flat}} (of dimension $l$). 

\begin{proposition}\label{pro:rat-pts2}
If $C$ is a rational boundary flat of $\BP(\sfT_N)$ of dimension $l$, there exists an element $g$ of $\UC(N)$ such that
$g(C)$ is contained in a boundary face of $\BP(\sfD_N)$ of dimension $l$.
\end{proposition}

\begin{proof} We prove this proposition by induction on the dimension $N\geq 3$. 
When $N=3$, the statement is equivalent to the previous proposition, because all boundary flats have dimension $0$.

Assume, now, that the proposition is proved up to dimension $N-1$.
Let $[u]\in C$ be a generic rational point, and let $h$ be an element of $\UC(N)$ that maps $[u]$
into $\sfD_N$. Since, $C$ is in the boundary of the Tits cone, so is $h(C)$. Projects $h(C)$ into $[c_1^\perp]$
from the vertex $[c_1]$. The image $\pi_1(h(C))$ is in the boundary of $\BP(\sfT_N)\cap [c_1^\perp]$; hence,
$\pi_1(h(C))$ is a rational boundary flat of dimension $l$ or $l-1$. Thus, there
exists an element $h'$ of $\UC(N)$ that stabilizes $[c_1]$ and maps $\pi_1(h(C))$ into the boundary of $\BP(\sfD_N)\cap [c_1^\perp]$
($\simeq \BP(\sfD_{N-1})$). 
The convex set generated by $h'(\pi_1(h(C)))$ and $[c_1]$ is also contained in a boundary face of $\BP(\sfD_N)$. 
This proves the proposition by induction.
\end{proof}


\section{Automorphisms, birational transformations, and the universal Coxeter group}


Our goal, in this section,  is to prove the first assertion of Theorem~\ref{main1}: We describe the groups of regular automorphisms and 
of birational transformations of generic Wehler varieties $X$. The group $\Bir(X)$ turns out to be isomorphic to $\UC(n+1)$ and its
action on the N\'eron-Severi group $\NS(X)$ is conjugate to the geometric representation of $\UC(n+1)$, as soon as $\dim(X)\geq 3$.

\subsection{Calabi-Yau hypersurfaces in Fano manifolds}
A Fano manifold is a complex projective manifold $V$ with ample anti-canonical bundle $K_V$.

\begin{theorem}\label{fano}
Let $n \ge 3$ be an integer and $V$ be a Fano manifold of dimension  $(n+1)$. 
Let  $M$ be a smooth member of the linear system $\vert -K_V \vert$. 
Let $\tau : M \rightarrow V$ be the natural inclusion. 
Then:
\begin{enumerate}
\item $M$ is a Calabi-Yau manifold of dimension $n \ge 3$. 

\item The pull-back morphism $\tau^* : {\rm Pic}\, (V) \to {\rm Pic}\, (M)$ is an isomorphism, and it 
 induces an isomorphism of ample cones:
 \[
\tau^*({\rm Amp}\, (V))={\rm Amp}\, (M). 
\]
\item $\Aut\,(M)$ is a finite group. 
\end{enumerate}
\end{theorem}

\begin{remark}\label{neron-severi}
For a Calabi-Yau manifold (resp. for a Fano manifold) $M$,  
the natural cycle map ${\rm Pic}\, (M) \rightarrow \NS\,(M)$ given by $L \mapsto c_1(L)$ is an isomorphism because $h^1(\mathcal O_M) = 0$. So, in what follows, we identify the Picard group ${\rm Pic}\, (M)$ and the N\'eron-Severi group $\NS\,(M)$.
\end{remark}
\begin{proof} By the adjunction formula, it follows that 
${\mathcal O}_M(K_M) \simeq {\mathcal O}_M$. By the Lefschetz hyperplane section theorem, $\pi_1(M) \simeq \pi_1(V) = \{1\}$, because every Fano manifold is simply connected. 
Consider the long exact sequence which is deduced from the exact sequence of sheaves 
\[
0 \to {\mathcal O}_V(-K_V) 
\to {\mathcal O}_V \to {\mathcal O}_M \to 0\]
and apply the Kodaira vanishing theorem to $-K_V$: It follows that
$h^k({\mathcal O_M}) = 0$ for $1 \le k \le n-1$; hence $h^0(\Omega_M^k) = 0$ for $1 \le k \le n-1$ by Hodge symmetry. This proves the assertion (1). 

The first part of assertion (2) follows from the Lefschetz hyperplane section theorem, because $n \ge 3$. By a result of Koll\'ar (\cite{Bo}, Appendix), the natural map $\tau_* : \overline{{\rm NE}}(V) \rightarrow \overline{{\rm NE}}(M)$ is an isomorphism. Taking the dual cones, we obtain the second part of assertion (2). 

The proof of assertion (3) is now classical (see \cite{Wi} Page 389 for instance). By assertion (1), $T_M \simeq \Omega_M^{n-1}$;
hence $h^0(T_M) = 0$. Thus $\dim \Aut\,(M) = 0$. 

The cone ${\rm Amp}\, (V)$, which is the dual of $\overline{{\rm NE}}(V)$, is a finite rational polyhedral cone because $V$ is a Fano manifold. As a consequence of assertion (2), ${\rm Amp}\, (M)$ is also a finite rational polyhedral cone. Thus, ${\rm Amp}\, (M)$ is the convex hull of a finite number of extremal rational rays $\R^{+} h_i$, $1 \le i \le \ell$, each $h_i$ being an integral primitive vector. The group $\Aut(M)$ preserves  ${\rm Amp}\, (M)$ and, therefore, permutes the $h_i$. From this, follows that the ample class  $h = \sum_{i=1}^{\ell} h_i$ is fixed by 
$\Aut\,(M)$. 
Let $L$ be the line bundle with first Chern class $h$.
Consider the embedding $\Theta_L\colon M \to \BP(H^0(M,L^{\otimes k})^\vee)$ which is defined by a  large enough multiple of $L$. Since $\Aut\, (M)$ preserves $h$, it acts by projective linear transformations on $\BP(H^0(M,L^{\otimes k})^\vee)$ and the embedding $\Theta_L$ is equivariant with respect to the action of $\Aut(M)$ on $M$, on one side, and   on $\BP(H^0(M,L^{\otimes k})^\vee)$, on the other side. The image of $\Aut(M)$ is the closed algebraic subgroup of ${\rm PGL}(H^0(M,L^{\otimes k})^\vee)$ that preserves $\Theta_L(M)$. Since $\dim \Aut\,(M) = 0$, this algebraic group  is finite, and assertion (3)  follows. 
\end{proof}

\subsection{Transformations of Wehler varieties}\label{par:WV}

\subsubsection{Products of lines}
Let $n$ be a positive integer. Denote
\begin{eqnarray*}
P(n+1) & := & ({\bbP}^1)^{n+1} = {\bbP}_1^1 \times {\bbP}_2^1 \times 
\cdots \times {\bbP}_{n+1}^1\\
P(n+1)_j & := &{\bbP}_1^1 \times \cdots {\bbP}_{j-1}^1 \times {\bbP}_{j+1}^1
\cdots \times {\bbP}_{n+1}^1 \simeq P(n)
\end{eqnarray*}
 and 
\begin{eqnarray*}
p^j & : &P(n+1) \to {\bbP}_j^1 \simeq {\bbP}^1\\
p_j &: &P(n+1) \to P(n+1)_j 
\end{eqnarray*}
the natural projections. Let $H_j$ be the divisor class of $(p^j)^{*}({\mathcal O}_{{\bbP}^1}(1))$. Then $P(n+1)$ is a Fano manifold of dimension $n+1$ that satisfies
\begin{eqnarray*}
\NS(P(n+1)) & = & \oplus_{j=1}^{n+1} {\mathbf Z} H_j\, , \\
-K_{P(n+1)} & = & \sum_{j=1}^{n+1} 2H_j\, ,\\
{\rm Amp}\,(P(n+1)) & = & \left\{ \sum a_i H_i\; \vert a_i \in \R_{>0} \; {\text{for all}} \;  i\right\},\, .
\end{eqnarray*}

\subsubsection{Wehler varieties}\label{par:wv}
 Let $X $  be an  element of the linear system $ \vert -K_{P(n+1)} \vert$; in other words,  $X$ is a   hypersurface of multi-degree $(2,2, \ldots, 2)$ in $P(n+1)$. More explicitly, for each index $j$ between $1$ and $n+1$,
the equation of $X$  in $P(n+1)$ can be written in the form
\begin{equation}\label{eq3}
F_{j, 1}x_{j, 0}^2 + F_{j, 2}x_{j,0}x_{j,1} + F_{j, 3}x_{j,1}^2 = 0
\end{equation}
where $[x_{j, 0} : x_{j, 1}]$ denotes the homogenous coordinates of 
${\bbP}_j^1$ and the $F_{j, k}$ ($1 \le k \le 3$) are homogeneous polynomial functions of multi-degree $(2,2, \ldots, 2)$ on $P(n+1)_j$. 

Let $\tau : X \to V$ be the natural inclusion and $h_j := \tau^*H_j$. If
$X$ is smooth, then $X$  is a Calabi-Yau manifold of dimension $n\geq 3$, and Theorem \ref{fano} (1) implies that 
\begin{equation}\label{eq4}
\NS\, (X) = \oplus_{j=1}^{n+1} {\mathbf Z} h_j\,\, {\text{ and }} \,\, {\rm Amp}\,(X) = \oplus_{j=1}^{n+1} \R_{>0} h_j\, .
\end{equation}
Let 
\[
\pi_j :=  p_j\circ \tau : X \to P(n+1)_j \simeq P(n)\,\, .
\]
This map is a surjective morphism of degree $2$; it is finite in the complement of
\[
B_j := \{ F_{j, 1} = F_{j,2} = F_{j,3} = 0\}\,\, .
\]
In what follows, we assume that each $B_j$ has  codimension $\ge 3$ and that $X$ is smooth. This is satisfied for a generic choice of $X\in \vert -K_{P(n+1)}\vert$. 

For $x \in B_j$, we have $\pi_j^{-1}(x) \simeq {\bbP}^1$. It follows that 
$\pi_j$ contracts no divisor. For $x \not\in B_j$, the set $\pi_j^{-1}(x)$ consists of $2$ points, say $\{y, y'\}$: The correspondence 
$y \leftrightarrow y'$ defines a birational involutive transformation $\iota_j$ of $X$ over $P(n+1)_j$. Thus, $\Bir(X)$ contains at least $n+1$
involutions $\iota_j$.

The group $\Bir\,(X)$ naturally acts on $\NS\, (X)$ as a group of linear automorphisms. Indeed, since $K_X$ is trivial, each element of $\Bir\,(X)$ is an isomorphism in codimension $1$ (see eg. \cite{Ka3}, Page 420). When $n=2$, $X$ is 
 a projective K3 surface. Since $X$ is a minimal surface, each $\iota_k$  is a (biregular) automorphism (see eg. \cite{BHPV}, Page 99, Claim). 

\subsection{The groups $\Aut(X)$ and $\Bir(X)$ in dimension $n\geq 3$}\label{par:AutBirN}
 In this section we prove the following strong version of the first assertions in Theorem \ref{main1}.

\begin{theorem}\label{wehlercy}
Let $n\geq 3$ be an integer.
Let $X\subset ({\BP}^1)^{n+1}$ be a generic Wehler variety. Let $\iota_j$ be the $n+1$ natural birational involutions of $X$. Then
\begin{enumerate}
\item  In the basis $( h_k )_{k=1}^{n+1}$ 
of $\NS\, (X)$, the matrix of $\iota_j^*$ coincides with $M_{n+1, j}$ (see Equation \eqref{eq:matrix}). 

\item The morphism 
\[
\Psi\colon \UC(n+1)\to \Bir(X)
\]
that maps the generators $t_j$ to the involutions $\iota_j$ is injective; the action of
$\,\Psi(\UC(n+1))$ on $\NS\, (X)$ is conjugate to the (dual of the) geometric representation of $\UC(n+1)$.

\item  The automorphism group of $X$ is trivial: $\Aut\, (X) = \{{\rm{Id}}_X \}$. 

\item  $\Bir\, (X) $ coincides with the subgroup $ \langle \iota_1\, ,\, \iota_2\, ,\, \cdots \iota_{n+1} \rangle \simeq \UC(n+1)$.
\end{enumerate}
\end{theorem}

\begin{remark}\label{smooth}
As our proof shows, Assertions (1) and (2) holds whenever $X$ is smooth and $n\geq 3$. They are also satisfied when $n=2$, if $\NS(X)$ is replaced by the subspace $\Z h_1\oplus \Z h_2\oplus\Z h_3$ (this subspace is invariant under the three involutions even if $\NS(X)$ has dimension $\geq4$).

Assertion (4) is certainly the most difficult part of this statement, and its proof makes use of delicate recent results in algebraic geometry.
\end{remark}

\begin{proof}[Proof of Assertions (1) and (2)]
By definition of $\iota_j$, we have $\iota_j^*(h_k) = h_k$ for $k \not= j$. Write $P(n+1) = P(n+1)_j \times {\bbP}_j^1$, where ${\bbP}_j^1$ is the $j$-th factor of $P(n+1)$. Let $(a, [b_0:b_1])$ be a point of $X \setminus \pi_j^{-1}(B_j)$, with $a\in P(n+1)_j$ and $[b_0:b_1]\in \bbP^1$. Then $\iota_j(a, [b_0:b_1])$ is the
second point  $(a, [c_0:c_1])$ of $X$ with the same projection $a$ in $P(n+1)_j$. The relation between the roots and the coefficients of the quadratic Equation \eqref{eq3} provides the formulas:
\[
\frac{c_0}{c_1} \cdot \frac{b_0}{b_1} = \frac{F_{j,3}(a)}{F_{j,1}(a)}\,\, {\text{ and }}\,\, \frac{c_0}{c_1} + \frac{b_0}{b_1} = -\frac{F_{j,2}(a)}{F_{j,1}(a)}\,\, .
\]
Here the polynomial $F_{j,3}$ is not zero and the divisors ${\rm div}\, (F_{j,k}\vert X)$ ($k = 1$, $2$, $3$) have no common component, because $X$ is smooth. Thus, in ${\rm Pic}\, (X) \simeq \NS\, (X)$, we obtain
\[
\iota_j^*(h_j) + h_j  = \sum_{k \not= j} 2h_k\,\, .
\]
This proves Assertion (1).

Since $\UC(n+1)$ is a free product of $(n+1)$ groups of order $2$, and its geometric representation is faithful, assertion (2) follows from assertion (1)
and the definition of the matrices $M_{N,j}$.
\end{proof}

\begin{proof}[Proof of Assertion (3)]
Let $x_j$ be the standard affine coordinate on ${\bbP}_j^1\setminus\{\infty\}$. Then $X \in \vert -K_{P(n+1)} \vert$ is determined by a polynomial function $f_X(x_1, \ldots, x_{n+1})$ of degree $\le 2$ with respect to each variable $x_j$. 

By the third assertion in Theorem (\ref{fano}), $\Aut\, (X)$ is a finite group. Since $\Aut\, (X)$ preserves ${\rm Amp}\, (X)$, it preserves the set $\{h_j\, \vert\, 1 \le j \le n+1\}$, permuting its elements.Thus, the isomorphism $H^0(\mathcal O_{P(n+1)}(H_j)) \simeq 
H^0(\mathcal O_{X}(h_j))$ implies that $\Aut(X)$ is a subgroup of $\Aut (P(n+1))$:
\[
\Aut\, (X) \subset \Aut\, (P(n+1)) = {\rm PGL}_2(\C)^{n+1} \rtimes {\mathfrak S}_{n+1}\,\, ,
\]
where ${\mathfrak S}_{n+1}$ denotes the group of permutations of the $(n+1)$ factors of $P(n+1)$.

The group ${\rm PGL}_2(\C)^{n+1} \rtimes {\mathfrak S}_{n+1}$ acts on $\vert -K_{P(n+1)} \vert$ and
$\Aut\, (X)$ coincides with the stabilizer of 
the corresponding point $X\in \vert -K_{P(n+1)} \vert$. 
Consider, for a generic $X$, the image $G$ of the morphism 
\[
\Aut\, (X) \to {\rm PGL}_2(\C)^{n+1} \rtimes {\mathfrak S}_{n+1}\to {\mathfrak S}_{n+1}\, .
\]
If $g$ is an element of $G$, and $X$ is generic, there is a lift $\tilde{g}_X\colon X\to X$ which is induced by
an automorphism of $P(n+1)$. This turns out to be  impossible, by considering the actions on 
the inhomogeneous quadratic monomials $x_j^2$ in the equation $f_X$ (for $1 \le j \le n+1$). 

Thus for $X$ generic, $\Aut\, (X) $ coincides with a finite subgroup of  $ {\rm PGL}_2(\C)^{n+1}$. 
Let ${\text{ Id}}_X \neq g$ be an automorphism of $X$; $g$ is induced by an element of  ${\rm PGL}_2(\C)^{n+1}$ of finite order. 
Then {\it up to conjugacy} inside $ {\rm PGL}_2(\C)^{n+1}$, the co-action 
of $g$ can be written as $g^*(x_j) = c_jx_j$ ($1 \le j \le n+1$), where 
$c_j$ are all roots of $1$ and at least one $c_j$, 
say $c_1$, is not $1$. 
By construction the equation $f_{X}(x_j)$ is $g^*$-semi-invariant: $f_X\circ g = \alpha(g) f_X$ for some root of unity $\alpha(g)$. Decompose 
$f_X$ into a linear combination of monomial factors  
\[
x_1^{k_1}x_2^{k_2} \cdots x^{k_{n+1}}\,\, ,\,\, k_j = 0, 1, 2\,\, ;
\]
the possible factors satisfy the eigenvalue relation
\[
c_1^{k_1}c_2^{k_2} \cdots c_{n+1}^{k_{n+1}} = \alpha(g)\, .
\]
Since $c_1 \not= 1$, for each fixed choice of $(k_2, k_3, \ldots, k_{n+1})$,  
at most two of the three monomials
\[
x_1^{2}x_2^{k_2} \cdots x^{k_{n+1}}\, ,\quad x_1x_2^{k_2} \cdots x^{k_{n+1}}\, ,\quad x_2^{k_2} \cdots x^{k_{n+1}}
\]
satisfy the relation above. 
Hence the number of monomial factors in $f_X(x_j)$ is at most $2\cdot3^{n}$. Moreover, given $g$, the possible values
of $\alpha(g)$ are finite (their number is bounded from above by the order of $g$).

Thus, the subset of varieties $X \in \vert -K_{P(n+1)} \vert$ with at least one automorphism $g \not= 1$ belongs to countably many subsets of dimension at most 
\[
2\cdot3^{n} -1+ \dim\, {\rm PGL}_2(\C)^{n+1}= 3^{n+1} + 3(n+1) - 3^n -1
\] 
in $\vert -K_{P(n+1)} \vert$. 
Since  $\dim\vert -K_{P(n+1)}\vert=3^{n+1}-1$, the codimension of these subsets is at least   
\[
3^{n} -3(n+1) \ge 3
\]
for $n \ge 2$. Thus, removing a countable union of subsets of $\vert -K_{P(n+1)} \vert$ of codimension $\ge 3$, the 
remaining generic members $X$ of $\vert -K_{P(n+1)} \vert$ have trivial automorphism group. 
\end{proof}

\begin{proof}[Proof of Assertion (4)] Let 
\[
X \stackrel{\overline{\pi}_j}{\to} \overline{X}_j \stackrel{q_j}{\to} P(n+1)_j
\] 
be the Stein factorization of $\pi_j$. 
Then, $\overline{\pi}_j$ is the small contraction corresponding to 
the codimension $1$ face $F_j := \sum_{k \not= j} \R_{\ge 0} h_k$ of the nef cone $\overline{{\rm Amp}}\, (X)$. Thus, $\rho(X/\overline{X}_j) = 1$ 
with a $\overline{\pi}_j$-ample generator $h_j$. Hence $\overline{\pi}_j$ is a flopping contraction of $X$. Let us describe the flop of $\overline{\pi}_j$. 

By definition of the Stein factorization, $\iota_j$ induces a {\it biregular} automorphism $\overline{\iota}_j$ of $\overline{X}_j$ that satisfies 
$\overline{\iota}_j \circ \overline{\pi}_j = \overline{\pi}_j \circ \iota_j$. 
We set 
$$\overline{\pi}_j^+ := (\overline{\iota}_j)^{-1} \circ \overline{\pi}_j : X \to \overline{X}_j\,\, .$$
Then, $\iota \circ \overline{\pi}_j^+ = \overline{\pi}_j$ 
and $\iota_j^*(h_j) = -h_j + \sum_{k \not= j} 2h_k$ by (1). 
In particular, $\iota_j^*(h_j)$ is $\overline{\pi}_j^+$-anti-ample. 
Hence $\overline{\pi}_j^+ : X \to \overline{X}_j$, or by abuse of language, 
the associated birational transformation $\iota_j$, is the flop of 
$\overline{\pi}_j : X \to \overline{X}_j$. 

Recall that any flopping contraction of a Calabi-Yau manifold is given by a codimension one face of $\overline{{\rm Amp}}\, (X)$ up to automorphisms of $X$ 
(\cite{Ka2}, Theorem (5.7)). Since there is no codimension one face of $\overline{{\rm Amp}}\, (X)$ other than the $F_j$ ($1 \le j \le n+1$), it follows that there is no flop other than $\iota_j$ ($1 \le j \le n+1$) up to $\Aut\,(X)$. On the other hand, by a fundamental result of Kawamata (\cite{Ka3}, Theorem 1), any birational map between minimal models is decomposed into finitely many flops up to automorphisms of the source variety. Thus any $\varphi \in \Bir\, (X)$ is decomposed 
into a finite sequence of flops modulo automorphisms of $X$. 
Hence $\Bir\, (X)$ is generated by $\Aut\,(X)$ and $\iota_j$ ($1 \le j \le n+1$). Since $\Aut\,(X) = \{{\text{Id}}_X\}$ for $X$ generic, assertion (4) is proved. \end{proof}

\begin{remark}\label{minimal} 
In general, the minimal models of a given variety are not unique up to 
isomorphisms. See for instance \cite{LO} for an example on a Calabi-Yau threefold. However, by the proof of Assertion (3), the generic Wehler variety $X$ has no other minimal model  than $X$ itself. 
\end{remark}

\subsection{Wehler surfaces}\label{par:WSurf}
\subsubsection{Statement}
Assume $n=2$, so that $X$ is now a smooth surface in $P(3)$. Then, $X$ is a projective K3 surface; in particular, $X$ is a minimal surface and each $\iota_k$  is a (biregular) automorphism (see eg. \cite{BHPV}, Page 99, Claim). Thus 
\[
\langle \iota_1, \iota_2, \iota_3 \rangle \subset \Aut(X).
\]
We now prove the following result, which contains a precise formulation of \S~\ref{par:intro-curvesurfaces}.
To state it, we keep the same notations as in Sections \ref{par:WV} and \ref{par:AutBirN}; in particular, $\tau$ is
the embedding of $X$ in $P(3)$ and the $h_j$ are obtained by pull-back of the classes $H_j$.

\begin{theorem}\label{wehler2} Let $X\subset \bbP^1\times \bbP^1\times \bbP^1$ be a smooth Wehler surface.
\begin{enumerate}

\item If $X$ is smooth, then 
$\Psi(\UC(3))\subset \Aut\,(X)$.

\item If $X$ is generic, then $\NS\, (X) = \oplus_{j=1}^{3} {\mathbf Z}h_j$. 

\item If $X$ is generic, then 
$\Aut\,(X) = \langle \iota_1, \iota_2, \iota_3 \rangle = 
\langle \iota_1 \rangle * \langle \iota_2 \rangle * \langle \iota_3 \rangle 
\simeq \UC(3)$.
\end{enumerate}
\end{theorem}

\begin{remark}
The N\'eron-Severi group may have dimension $\rho(X)>3$; such a jump of $\rho(X)$ can be achieved by arbitrary small 
deformations (see \cite{Og}), and it may leads to very nice examples of ample cones (see \cite{Bar}).
\end{remark}
\subsubsection{Proof of Theorem~\ref{wehler2}}
The three following lemmas prove Theorem~\ref{wehler2}.
In this section, $X$ is a smooth Wehler surface. We start with a description of $\NS(X)$ and of the quadratic form defined by the intersection of divisor classes.

\begin{lemma}\label{lemm61} If $X$ is generic, then  $\NS\, (X) = \oplus_{i=1}^{3} {\mathbf Z} h_i$. 
The matrix $((h_i\cdot h_j)_X)$ of the intersection form on $\NS(X)$ is
\[
((h_i.h_j)_S) = \left(\begin{array}{rrr}
0 & 2 & 2\\
2 & 0 & 2\\
2 & 2 & 0
\end{array} \right).\]
\end{lemma}

\begin{proof} Since $X$ is generic, it follows from the Noether-Lefschetz theorem (\cite{Vo}, Theorem (3.33)) that $\tau^* : \NS\, (P(3)) 
\rightarrow  \NS\, (X)$ is an isomorphism. This proves (1), and (2) follows from $(h_i.h_j)_{S} = (H_i.H_j.2(H_1 + H_2 + H_3))_{P(3)}$. 
\end{proof}

\begin{remark} More generally, if $W$ is a generic element of $\vert -K_V \vert$ of a smooth Fano threefold with very ample anti-canonical divisor $-K_V$, 
then $W$ is a K3 surface and $\NS\, (V) \simeq \NS\, (W)$ under the natural inclusion map. \end{remark}

Thus, the intersection form corresponds to the quadratic form $b_3$ on $V_3$ that is preserved by $\UC(3)$. The fact that $b_3$ has signature
$(1,2)$ is an instance of Hodge index theorem (see Remark \ref{eg:N=3}).

Once we know the intersection form, one can check that there is no effective curve with negative self-intersection on $X$. Indeed, if there were such a curve, one could find an irreducible curve $E\subset X$ with $E\cdot E<0$; the genus formula would imply that $E$ is a smooth rational curve with self-intersection $-2$; but the intersection form does not represent the value $-2$. 

It is known that the ample cone is the set of vectors $u$ in $\NS(X;\R)$ such that $u\cdot u >0$, $u\cdot E >0$ for all effective curves, and $u\cdot u_0>0$ for a given ample class (for example $u_0=\sum h_j$). Since there are no curves with negative self-intersection, one obtains:

\begin{lemma}\label{lemma61bis} If $X$ is generic, the ample cone ${\rm Amp}\, (X)$  coincides with the positive cone 
\[
{\sf{Pos}}(X)=\{u\in \NS(X;\R)\; \vert \; u\cdot u >0 \; {\text{and}} \; u\cdot h_1>0 \}.
\] \end{lemma}

\begin{remark} Even though $X$ is generic and 
$\tau^* : \NS\, (P(3)) \to \NS\, (X)$ is an isomorphism, 
the image of $\tau^* : {\rm Amp}\, (P(3)) \to 
{\rm Amp}\, (X)$ is strictly smaller than ${\rm Amp}\, (X)$, in contrast with the higher dimensional case (see Theorem (\ref{fano})).\footnote{This gives an explicit, negative, answer for a question of Pr. Yoshinori Gongyo to K.~Oguiso.}
\end{remark}

From the previous sections (see Remark~\ref{smooth}), we know that, for all smooth Wehler surfaces,
\begin{itemize}
\item the subspace $N_X := \oplus_{j=1}^{3} \mathbf Z h_j$ is invariant 
under the action of $\langle \iota_1, \iota_2, \iota_3 \rangle$  on $\NS(X)$;
\item the matrix of $(\iota_k^*)_{\vert N_X}$ in the basis $(h_1, h_2, h_3)$ is equal to $M_{3, k}$, where  $M_{3, k}$ is defined in \S~\ref{matrix};
\item there are no non-obvious relations between the three involutions $\iota_k^*$, hence 
\[
\langle \iota_1 \rangle * 
\langle \iota_2 \rangle * \langle \iota_3 \rangle \simeq \UC(3)
\]
\end{itemize}
Moreover, there is a linear isomorphism from the (dual of the) geometric representation $V_3$ to 
$N_X$ which conjugates the action of $\UC(3)$ with the action of $ \langle \iota_1, \iota_2 , \iota_3 \rangle$, and 
maps $\sfD_3$ to the convex cone 
\[
\Delta= \R_+h_1\oplus \R_+h_2\oplus \R_+h_3,
\] 
the quadratic form $b_3$ to the intersection form on $N_X$, and the Tits cone $\sfT_3$ to the positive cone ${\sf{Pos}}(X)$.
Since the ample cone is invariant under the action of $\Aut(X)$ and contains $\R_+h_1\oplus \R_+h_2\oplus \R_+h_3$, 
this gives another proof of Lemma~\ref{lemma61bis}.

\begin{lemma}\label{lem63} If $X$ is generic, then
\begin{enumerate}
\item no element of $\Aut\, (X)\smallsetminus \{{\rm Id}_X\}$ 
is induced by an element of $\Aut\, (P(3))$. 

\item  $\Aut\,(X) = \langle \iota_1, \iota_2 , \iota_3 \rangle$. 
\end{enumerate}
\end{lemma}
\begin{proof} 
The proof of (1) is the same as for Theorem \ref{wehlercy}. 
Let us prove (2). Since $X$ is generic, $\NS(X)=N_X$. The image $G$ of $\Aut(X)$ in $\GL(\NS(X))$.
contains the group generated by the three involutions $\iota_j^*$; as such it as finite index in the group of
isometries of the lattice $\NS(X)$ with respect to the intersection form (see \S~\ref{par:T3}). Thus, if $G$ is larger than
$\langle \iota_1^*, \iota_2^*, \iota_3^*\rangle$, there exists an element $g^*$ of $G\smallsetminus\{{\rm Id}\}$ that preserves the fundamental 
domain $\Delta$. Such an element permutes the vertices of $\Delta$. As in the proof of Theorem \ref{wehlercy}, one sees
that $g^*$ would be induced by an element of $\Aut(P(3))$, contradicting Assertion (1).
Thus, $G$ coincides with $\langle \iota_1^*, \iota_2^*, \iota_3^*\rangle$. On the other hand, if $f\in \Aut(X)$ acts trivially 
on $\NS(X)$ then, again, $f$ is induced by an element of $\Aut(P(3))$. Thus, Assertion (2) follows from Assertion (1).
\end{proof}


\section{The movable cone}


To conclude the proof of Theorem \ref{main1}, we need to describe the movable cone of 
Wehler varieties. This section provides a proof of a more 
explicit result, Theorem~\ref{movable}. 

\subsection{Statement}\label{par:4.1}
To state the main result of this section, we implicitly identify the geometric representation $V_{n+1}$ 
of $\UC(n+1)$ to its dual $V_{n+1}^*$; for this, we make use of the duality offered by the non-degenerate quadratic 
form $b_{n+1}$, as in Section~\ref{par:UCN}. This said, let $\rho\colon \UC(n+1)\to \GL(V_{n+1})$ be the
geometric representation of the universal Coxeter group $\UC(n+1)$. Let 
\[
\Psi\colon \UC(n+1)\to \Bir(X)
\]
be the isomorphism that maps the generators $t_i$ of $\UC(n+1)$ to the generators $\iota_i$ of $\Bir(X)$. 

\begin{theorem}\label{movable} Let $n\geq 3$ be an integer, and let $X\subset (\BP^1)^{n+1}$ be
a generic Wehler variety of dimension $n$. 

There is a linear isomorphism $\Phi\colon V_{n+1}\to \NS(X)$ such that 
\begin{enumerate}
\item 
$
\Phi\circ\rho(w)= \Psi(w)^*\circ \Phi
$
for all elements $w$ of $\UC(n+1)$;
\item the fundamental domain $\sfD_{n+1}$ of $\UC(n+1)$ is mapped onto the nef cone $\overline{{\rm Amp}}\,(X)$ by $\Phi$;
\item the Tits cone $\sfT_{n+1}\subset V$ is mapped onto the movable  effective cone ${\mathcal M}^{e}(X)$ by $\Psi$.
\end{enumerate}
In particular, the nef cone is a fundamental domain for the action of $\Bir(X)$ on the movable effective cone.
\end{theorem}

\begin{remark}\footnote{This remark provides an explicit negative 
answer to the question asked by Mr Y.~Gongyo (in any dimension $\ge 3$) to K.~Oguiso.}
\label{cone}
Since $\Bir\, (X)$ is much bigger than $\Aut\, (X)$, 
Theorem \ref{movable} implies that the movable effective cone ${\mathcal M}^e\, (X)$, whence the pseudo effective cone $\overline{\mathcal B}\,(X)$, is much bigger than the nef cone $\overline{{\rm Amp}}\,(X)$. On the other hand, for the ambient space $P(n+1)$, we have 
\[
\overline{\mathcal B}\,(P(n+1)) = \overline{\rm Amp}\,(P(n+1)).
\] 
This is a direct consequence of the fact that the intersection numbers 
\[
(v.H_1 \cdots H_{k-1} \cdot H_{k+1} \cdots H_{n+1})_{P(n+1)}\,\, ,\,\, 
1 \le k \le n+1\, ,
\]
are non-negative if $v \in \overline{\mathcal B}\,(P(n+1))$.
So, under the isomorphism $\NS\, (P(n+1)) \simeq \NS\, (X)$, we have
\begin{itemize}
\item  $\overline{{\rm Amp}}\,(P(n+1)) \simeq \overline{{\rm Amp}}\,(X)$ 
(see Theorem~\ref{fano})
\item $\overline{\mathcal B}\,(P(n+1)) \not\simeq \overline{\mathcal B}\,(X)$, 
even if $X$ is generic. 
\end{itemize} 
\end{remark}

In the rest of this section, we prove Theorem~\ref{movable}. Thus, in what follows, $X$ is a generic Wehler 
variety of dimension $n\geq 3$.

\subsection{Proof}\label{par:proofmovable}

Let $\Psi$ be the isomorphism described in Section \ref{par:4.1}, and $\Phi$ the linear map which applies the cone $\sfD_{n+1}$ 
onto the nef cone  $\overline{{\rm Amp}}\,(X)$, mapping each vertex $c_j$ to $h_j$. 
With such a choice, Assertions (1) and (2) are part of Section~\ref{par:AutBirN}, and we only  
need to prove Assertion (3), i.e. that the Tits cone $\sfT_{n+1}$
is mapped bijectively onto the movable effective cone ${\mathcal M}^{e}(X)$ by $\Phi$. 

\begin{lemma}\label{lem1} Let $X$ be a generic Wehler variety of dimension $n\geq 3$. Then
\begin{enumerate}
\item $\overline{{\rm Amp}}\,(X) \subset {\mathcal M}^{e}(X)$;

\item $g^*({\rm Amp}\,(X)) \cap {\rm Amp}\,(X) = \emptyset$ for all
$g\neq {\rm{Id}}_X$ in $\Bir(X)$.
\end{enumerate}
\end{lemma}

\begin{proof} Since the divisor classes $h_j$ ($1 \le j \le n+1$) are free, they are movable; assertion  
(1) follows from the fact that the nef cone is generated, as a convex cone, by these classes. If $g \in \Bir\, (X)$ satisfies $g^*({\rm Amp}\,(X)) \cap {\rm Amp}\,(X) \not= \emptyset$, then $g \in \Aut\, (X)$ (see \cite{Ka2}, Lemma~1.5). Thus, Assertion (2) follows from $\Aut\, (X) = \{  {\text{Id}}_X\}$. 
\end{proof}

Since the movable effective cone is $\Bir(X)$-invariant, the orbit of $\overline{{\rm Amp}}\,(X)$  is contained in $ {\mathcal M}^{e}(X)$. 
Hence, 
\[
\Phi(\sfT_{n+1})\subset  {\mathcal M}^{e}(X)
\]
and we want to show the reverse inclusion (note that the closures of these two sets are equal).

\begin{lemma}\label{lem2} Let $X$ be a generic Wehler variety of dimension $n\geq 3$.
For any given effective integral divisor class $D$, there is a birational transformation
$g$ of $X$ such that $g^*(D) \in \overline{{\rm Amp}}\,(X)$. That is, $D$ is contained in $\Phi(\sfT_{N+1})$.
\end{lemma}

The proof is similar to the proof of Proposition~\ref{pro:rat-pts1}. Here, one makes use of the intersection form
and positivity properties of effective divisor classes, instead of the quadratic form $b_N$. 

\begin{remark}
We do not know any geometric interpretation of the quadratic form $b_{n+1}$ on $\NS(X)$ for $n\geq 3$.
\end{remark}

\begin{proof} 
Define $D_1 := D$. In $\NS\, (X)$, we can write 
\[
D_1 = \sum_{j=1}^{n+1} a_j(D_1)h_j\,\, ,
\]
where the coefficients $a_{j}(D_1)$ are integers. Put 
\[
s(D_1) := \sum_{j=1}^{n+1} a_{j}(D_1)\,\, .
\]
By definition, $s(D_1)$ is an integer. 
Since $D_1$ is an effective divisor class and the classes $h_i$ are nef, it follows that 
\[
a_{n}(D_1) + a_{n+1}(D_1) = (D\cdot h_1\cdot h_2 \cdots h_{n-1})_{X} \geq 0\, .
\]
For the same reason, 
$a_i(D_1) + a_j(D_1) \geq 0$ for all $1 \leq i \not= j \le n+1$. 
Hence there is {\it at most} one $i$ such that $a_i(D_1) <0$.

Moreover the sum $s(D_1)$ is non-negative. Indeed, there is at most one negative term,
 say $a_1(D_1)$, in the sum defining $s(D_1)$;
 since
\[
s(D_1) = (a_1(D_1) + a_2(D_1)) + a_3(D_1) + 
\cdots + a_{n+1}(D_1) \ge 0\,\, 
\]
and $a_1(D) + a_2(D_1) \geq 0$, this shows that $s(D_1)$ is non-negative.
 
If $D_1 \in \overline{{\rm Amp}}\, (X)$, then we can take $g = 1$. So, we may assume that $D_1$ is not in $\overline{{\rm Amp}}\, (X)$, which means that there is a unique index $i$ with $a_{i}(D_1) < 0$. Then, consider the new divisor class $D_2 := \iota_i^*(D_1)$; it is effective. By definition 
of $a_j(\cdot)$ and Theorem~\ref{wehlercy} (first assertion), we have
\begin{eqnarray*}
 D_2 & = & \sum_{j=1}^{n+1} a_j(D_2)h_j  \\
 & = & -a_i(D_1)h_i + \sum_{j \not= i} (a_j(D_1) + 2a_i(D_1))h_j\,\, .
\end{eqnarray*}
Computing $s(D_2) = \sum  a_j(D_2) $, the inequality  $a_i(D_1) < 0$ provides
\begin{eqnarray*}
s(D_2) 
& = & -a_i(D_1) + \sum_{j \not= i} (a_j(D_1) + 2a_i(D_1)) \\
& = & s(D_1) + 
(2n-1)a_i(D_1) \\
& < &  s(D_1)\,\, .
\end{eqnarray*}
If $a_j(D_2) \ge 0$ for all $j$, then $D_2 \in \overline{{\rm Amp}}\, (X)$, 
and we are done. Otherwise, there is $j$ such that $a_j(D_2) < 0$. Consider 
then the divisor class $D_3 := \iota_j^*(D_2)$. As above, 
$D_3$ is an effective 
divisor class such that $s(D_3) < s(D_2)$.

We repeat this process: At each step, the sum $s(\cdot)$ decreases by at least one unit. 
Since $s(\cdot)$ is a positive integer for all effective divisors, the process stops, and provides
an effective divisor $D_k$ in the $\Bir(X)$-orbit of $D_1$ such that all coefficients $a_i(D_k)$
are non-negative, which means that $D_k$ is an element of $\overline{{\rm Amp}}\, (X)$.
\end{proof}

Let $u$ be an element of the movable cone ${\mathcal M}^e(X)$. As ${\mathcal M}^e(X)$ is the intersection of ${\overline{  {\mathcal M}}} (X)   $
and ${\mathcal B}^e\, (X)$, we can write 
\[
u=\sum_{i=1}^{l+1} r_i D_i
\]
where each $D_i$ is an effective divisor class and all $r_i$ are positive real numbers. 
The Tits cone $\sfT_{N+1}$ is a convex set (see \S~\ref{par:titsconegene}) and the previous lemma shows that each $D_i$ is in the image
of $\sfT_{N+1}$; since $u$ is a convex combination of the $D_i$, $u$ is an element of $\Phi(\sfT_{N+1})$. Consequently, we obtain 
\[
\Phi(\sfT_{n+1})={\mathcal M}^e(X),
\]
and Theorem~\ref{movable} is proved. 

\subsection{Complement} The following theorem follows from our description of the movable cone as (the image of) the Tits
cone $\sfT_{N+1}$ and the description of rational boundary points of $\sfT_{N+1}$ obtained in Section~\ref{par:Cox-ratio-points}. As far as we know, this kind of statement is not predicted by the original cone conjecture.

\begin{theorem}
Let $n \ge 3$ be an integer. Let $X$ be a generic hypersurface of multi-degree $(2, \ldots, 2)$ in $({\bbP^1})^{n+1}$.
Let $D$ be rational boundary point of the movable cone $\overline{{\mathcal M}}\, (M)$. Then there exists a pseudo-automorphism  $f$ of $X$ such that $f^*(D)$ is in the nef cone $\overline{{\rm Amp}}\, (M)$. Hence, rational boundary points of $\overline{{\mathcal M}}\, (M)$ are movable. \end{theorem}

\section{Universal cover of Hilbert schemes of Enriques surfaces}\label{par:main2}
In this section, we shall prove Theorem~\ref{main2} and a few refinements. 
\subsection{Enriques surfaces (see eg. \cite{BHPV}, Chapter VIII)} 
An Enriques surface $S$ is a compact complex surface 
whose universal cover  $\tilde{S}$ is a K3 surface. The Enriques surfaces 
form a $10$-dimensional family; all of them are
projective and their fundamental group is isomorphic to $\Z/2\Z$. 

Let $S$ be an Enriques surface. The free part of the N\'eron-Severi group $\NS_f\, (S)$ 
is isomorphic to the lattice $U \oplus E_8(-1)$, were (cf. \cite{BHPV}, Chapter I, Section 2)
\begin{itemize}
\item $U$ is the unique even unimodular lattice of signature $(1,1)$ ;
\item  $E_8(-1)$ is the unique even unimodular {\it negative} definite lattice of rank~$8$. 
 \end{itemize}
From now on, we identify the lattices 
$\NS_f\, (S)$ and $U \oplus E_8(-1)$. We denote the group of isometries 
of $\NS_f\, (S)$ preserving the positive cone by 
${\rm O}^{+}_{10}$. Here, by definition, the positive cone ${\sf Pos}(S)$ is the connected component of $\{\, x \in \NS\, (S)_\R\, \vert\, (x^2) > 0\}$, containing the ample cone. We define
\[
{\rm O}^{+}_{10}[2] = \{ \varphi \in {\rm O}^{+}_{10}\, \vert\, \varphi = \; {\rm{Id}} \mod(2)\}  \, . 
\]

\begin{theorem}[see \cite{BP}, Theorem 3.4, Proposition 2.8]\label{enriques} 
Let $S$ be a generic Enriques surface. Then $S$ does not contain any smooth rational curve and 
the morphism $g\mapsto g^*\in \GL(\NS_f(S))$ provides an isomorphism
\[
\Aut\, (S) \simeq {\rm O}^{+}_{10}[2].
\]
\end{theorem}

Moreover, ${\rm O}^{+}_{10}$ is isomorphic to the Coxeter group associated to the Coxeter diagram ${\rm T}_{2,3,7}$, a tree with $10$ vertices and three branches of length $2$, $3$, and $7$ respectively (see \cite{CoDo}).  
From now on, $S$ is a generic Enriques surface. 

\subsection{Hilbert schemes and positive entropy}

The group $\Aut(S)$ acts by automorphisms on the Hilbert scheme ${\rm Hilb}^{n}(S)$; this defines a morphism 
\[
\rho_n\colon \Aut(S)\to \Aut({\rm Hilb}^{n}(S)).
\]
 Denote by ${\mathcal{H}}^n(S)$ the universal cover $\widetilde{{\rm Hilb}^{n}(S)}$, and by ${\rm A}_n(S)$ the group of all automorphisms of ${\mathcal{H}}^n(S)$ which are obtained by lifting elements of $\rho_n(\Aut(S))$. By construction, there is an exact sequence 
 \[
 1\to \Z/2\Z \to {\rm A}_n(S) \to \rho_n(\Aut(S))\to 1.
 \]
Up to conjugacy, the group ${\rm O}^{+}_{10}[2]$ is a lattice in the Lie group ${\rm O}_{1,9}(\R)$; as such, it contains a non-abelian free group (we shall describe an explicit free subgroup below). Since all elements in the kernel of $\rho_n$ have finite order, the group ${\rm A}_n(S)$ is commensurable to a lattice in the Lie group ${\rm O}_{1,9}(\R)$ and contains a non-abelian free group. Since all holomorphic vector fields on ${\mathcal{H}}^n(S)$ vanish identically, one can find such a free group that acts faithfully on the N\'eron-Severi group of ${\mathcal{H}}^n(S)$. 

\begin{lemma}\label{lem:free}
Let $M$ be a complex projective manifold. Let $G$ be a subgroup of $\Aut(M)$ such that (i) $G$ is a non-abelian free group and (ii) $G$ acts faithfully on $\NS(M)$. Then there exists an element $g$ in $G$ such that $g^*\colon \NS(M)\to \NS(M)$ has an eigenvalue $\lambda$ with $\vert \lambda \vert >1$. 
In particular, the topological entropy of $g$ is strictly positive.
\end{lemma}

\begin{proof}
The first assertion follows from the fact that $G$ is a subgroup of $\GL(\NS(X))$, hence of $\GL_m(\Z)$, where $m$ is the Picard number of $M$: If all eigenvalues of all elements $g^*$ had modulus one, they would be roots of $1$, and a finite index subgroup of $G$ would be solvable (conjugate to 
a subgroup of $\GL_m(\R)$ made of upper triangular matrices with coefficients $1$ on the diagonal). The second follows from Yomdin's lower bound for
the topological entropy.\footnote{An element $g \in \Aut\,(M)$ is of positive entropy if and only if the spectral radius of the action of $g$ on $H^{1,1}(M,\R)$ is strictly bigger than $1$. The  entropy $h_{top}(g)$  is equal to the spectral radius of $g^*$ on $\oplus_p H^{p,p}(M,\R)$. (see \cite{Gro})}
\end{proof}

Thus ${\mathcal{H}}^n(S)$ has many automorphisms with positive entropy. In the next sections, we make this statement more precise and explicit, by constructing an embedding of $\UC(3)$ into ${\rm A}_n(S)$.

\subsection{Three involutions and positive entropy}

Let $e_1$, $e_2$ be a standard basis of $U$, that is 
$U = \ Ze_1\oplus \Z e_2 $ and 
\[ (e_1^2)_S = (e_2^2)_S = 0, \quad (e_1\cdot e_2)_S = 1.
\]
 We can, and do choose $e_1$ and $e_2$  in the closure of the 
positive cone ${\sf Pos}(S)$. 

Let $v \in E_8(-1)$ be an element such that 
$(v^2)_S = -2$. Put $e_3 := e_1 + e_2 + v$. Then, we have 
\[
(e_3^2)_S = 0\,\, ,\,\, (e_3\cdot e_1)_S = (e_3\cdot  e_2)_S = 1\, .
\]
So, $e_2$ and $e_3$ are also in the closure of the positive cone 
and the three sublattices 
\[
U_3 = \Z e_1\oplus \Z e_2 \, , \; U_2 =\Z e_1\oplus \Z e_3 \, , \;  U_1 = \Z e_2\oplus \Z e_3 
\]
of $\NS_f\, (S)$ are isomorphic to $U$. 

Let $j$ be one of the indices $1,2,3$. The sublattice $U_j$ determines  an orthogonal decomposition  
$\NS_f(S) = U_j \oplus U_j^{\perp}$.
Consider the  isometry of $\NS_f(S)$ defined by
\[
\iota_j^{*} = {\rm id}_{U_j} \oplus -{\rm id}_{U_j^{\perp}}
\]
Then, $\iota_j^{*}$ is an  element of  ${\rm O}^{+}_{10}[2]$, and  is induced by a unique automorphism $\iota_j$ of $S$, and $\iota_j$ is an involution (apply Theorem~\ref{enriques}). 
Let $\langle \iota_1 , \iota_2 , \iota_3\rangle$ be the subgroup of $\Aut\, (S)$ generated by $\iota_1$, $\iota_2$ 
and $\iota_3$. 

\begin{theorem}\label{enriquesautom} Let $S$ be a generic Enriques surface. There are no non-trivial relations between the three
linear transformations $\iota^*_k$: The groups
$\langle \iota_1 , \iota_2 , \iota_3\rangle$, $ \langle \iota_1^* \rangle * 
\langle \iota_2^* \rangle * \langle \iota_3^* \rangle $ and  $\UC(3)$ are isomorphic.
Moreover, the maximal eigenvalue of $\iota_1^*\iota_2^*\iota_3^*$ 
on $\NS\,(S)$ is equal to $9+4\sqrt{5}$. In particular, the topological entropy of the automorphism $\iota_3\circ\iota_2\circ\iota_1$ 
is  positive.   
\end{theorem}

\begin{proof} 
Consider the sub-lattice of $\NS(S)$ defined by  
\[
L = \Z  e_1\oplus \Z e_2\oplus \Z v  = \Z  e_1\oplus \Z e_2\oplus \Z e_3.
\]
By construction $\iota_1^*(e_2) = e_2$ and $\iota_1^*(e_3) = e_3$. Since $2e_2+v$ is orthogonal to both  $e_2$ and $e_3$, it follows that 
$2e_2 + v \in U_1^{\perp}$. Thus $\iota_1^*(2e_2 + v) = -(2e_2 + v)$ and $\iota_1^*(v) = -4e_2 - v$. One deduces easily that
$\iota_1^*(e_1) = -e_1 + 2e_2 + 2e_3$.
Permuting the indices, we obtain a similar formula for $\iota^*_2$ and $\iota^*_3$. It follows that
the lattice  $L$ is $\langle \iota_1 , \iota_2 , \iota_3\rangle$-invariant, and  
 the matrices of the involutions in the basis $(e_1, e_2, e_3)$ are
$$\iota^*_{3 \vert L} = M_{3,3}\,\, ,\,\, \iota^*_{2 \vert L} = M_{3,2},\, ,\,\, \iota^*_{1 \vert L} = M_{3,1}\,\, ,$$
($M_{3,3}$, $M_{3,2}$ and $M_{3,1}$ are the matrices introduced in  
Section \ref{matrix}). 

The remaining assertions follow from the fact that the maximal eigenvalues of 
the product matrix $M_{3,1}M_{3,2}M_{3,3}$ is $9 + 4\sqrt{5} > 1$. 
\end{proof}

The following remark, which has been kindly communicated to us by professor Shigeru Mukai, provides a geometric explanation of the previous statement, which is
related to Wehler surfaces. 
\begin{remark}\label{classical} 
The class $h = e_1 + e_2 + e_3$ determines an ample line bundle of degree $6$ 
and the projective model of $S$ associated to $\vert h \vert$ is a sextic surface in $\bbP^3$ singular along the $6$-lines of a tetrahedron. Then the universal cover $\tilde{S}$ has a projective model of degree $12$: It is a quadratic section of the Segre manifold $P(3) = \bbP^1 \times \bbP^1 \times \bbP^1 \subset \bbP^{7}$; so, $\tilde{S}$ is a 
K3 surface of multi-degree 
$(2,2,2)$ in $P(3)$, i.e., a Wehler surface.

Equivalently, let $\pi : \tilde{S} \to S$ be the universal cover of $S$. The classes $\pi^*e_i$  define three different elliptic fibrations $\varphi_i : \tilde{S} \to {\bbP}^1$ with no reducible fiber. Hence $\varphi_1 \times \varphi_2 \times \varphi_3$  embeds  $\tilde{S}$ into 
$P(3)$ and the image is a surface of multi-degree $(2,2,2)$. 
 
Then, one easily shows that the action of $\UC(3)$ on $S$ is covered by the natural action of $\UC(3)$ on the Wehler surface $\tilde{S}$ (Theorem~\ref{wehler2}-(2) in \S 6).
\end{remark}

\subsection{Proof of Theorem~\ref{main2}}
 We have a natural biregular action of the group $\UC(3) \simeq \langle \iota_1, \iota_2, \iota_3 \rangle$ on the Hilbert scheme 
${\rm Hilb}^{n}(S)$, induced by the action on $S$. Each $\iota_j \in \Aut\, ({\rm Hilb}^{n}(S))$ lifts equivariantly to a biregular action $\tilde{\iota}_j \in \Aut\, ({\mathcal{H}}^n(S))$ on the universal cover ${\mathcal{H}}^n(S)=\widetilde{{\rm Hilb}^{n}(S)}$. 
It suffices to show that each $\tilde{\iota}_j$ is an involution (we will then have natural surjective  homomorphisms
$\UC(3) \to \langle \tilde{\iota}_1, \tilde{\iota}_2, \tilde{\iota}_3 \rangle \to \langle \iota_1, \iota_2, \iota_3 \rangle \simeq \UC(3)$,
hence all the arrows will be isomorphic)

\begin{lemma}[see \cite{MN}]
Let $s$ be a holomorphic involution of ${\rm Hilb}^{n}(S)$. Then, all lifts $\tilde{s}$ of $s$ to the universal cover ${\mathcal{H}}^n(S)$
are involutions too. 
\end{lemma}

\begin{proof}
Let $\pi \colon {\mathcal{H}}^n(S) \to {\rm Hilb}^{n}(S)$ be the universal covering map and $\sigma$ be the covering involution. Let $\tilde{s}$ 
be a lift of $s$, {\sl{i.e.}} an automorphism of ${\mathcal{H}}^n(S)$ with $\pi\circ \tilde{s} = s\circ \pi$.
Then $(\tilde{s})^2$ is either the identity map or $\sigma$. Assume $(\tilde{s})^2 = \sigma$ and follow \cite{MN}, Lemma (1.2), to derive a contradiction. By assumption, $\langle \tilde{s} \rangle$ is a cyclic group of order $4$ and acts freely on ${\mathcal{H}}^n(S)$ because so does $\sigma$. 
By cite{OS}, Theorem (3.1), 
Lemma (3.2), the Euler characteristic of the sheaf ${\mathcal{O}}_{\mathcal{H}}^n(S)$ is equal to $2$.  Since the group $\langle \tilde{s} \rangle$  has order $4$ and acts freely, the quotient ${\mathcal{H}}^n(S) /\langle \tilde{\iota}_j \rangle$ is smooth and 
\[
\chi \left(\mathcal O_{{\mathcal{H}}^n(S)/\langle \tilde{s}\rangle}\right) = \frac{1}{4}\chi  \left(\mathcal O_{{\mathcal{H}}^n(S)}\right) = \frac{1}{2}\, ,
\]
a contradiction because this characteristic should be an integer.  \end{proof}

This completes the proof of Theorem~\ref{main2}~: one gets an embedding of $\UC(3)$ in $\Aut({\mathcal{H}}^n(S))$, and its image contains automorphisms with positive entropy (either by Theorem~\ref{enriquesautom} or Lemma~\ref{lem:free}).

\subsection{A question} There are Enriques surfaces with
$\vert \Aut\, (S) \vert < \infty$; on the other hand, Kondo shows that $\vert \Aut\, (\tilde{S}) \vert = \infty$ for {\it every} K3 surface $\tilde{S}$ which is the universal cover of an Enriques surface (see \cite{Kn} for these facts). Thus, it is natural to ask the following question.
\begin{question}\label{infinite}
$\vert \Aut\, (\widetilde{{\rm Hilb}^{n}(S)}) \vert = \infty$ for {\it every} Enriques surface?
\end{question}

 {\bf Acknowledgement.} The main idea of this work was found by the second author after fruitful discussions with Professors Klaus Hulek and Matthias Schu\"ett and with Doctor Arthur Prendergast-Smith, during his stay at Hannover in May 2011. K. Oguiso would like to express his best thanks to all of them for discussions and to Professors K. Hulek and M. Schu\"ett for invitation. We are also very grateful to Professors H. Esnault, Akira Fujiki, Yoshinori Gongyo, Shu Kawaguchi, J\'anos Koll\'ar, James McKernan and Shigeru Mukai for many valuable comments 
relevant to this work.

\end{document}